\documentclass[12pt,a4paper, reqno]{article}
\usepackage{amsmath}
\usepackage{amsfonts}
\usepackage{amssymb}
\usepackage{color}
\usepackage{comment,graphicx}
\usepackage[T1]{fontenc}
\usepackage{url}
\usepackage{ae}

\def\thebibliograph#1#2{\section*{{\normalsize \bf #2}}\list
{[\arabic{enumi}]}{\settowidth\labelwidth{[#1]}\leftmargin\labelwidth
\advance\leftmargin\labelsep
\usecounter{enumi}}
\def\newblock{\hskip .11em plus .33em minus -.07em}
\sloppy
\sfcode`\.=1000\relax}

\numberwithin{equation}{section}
\newtheorem{theorem} {Theorem} [section]

\newtheorem{definition}{Definition} [section]
\newtheorem{corollary}{Corollary} [section]
\newtheorem{lemma}{Lemma} [section]
\newtheorem{remark}{Remark} [section]

\begin{document}

\title{Variable Besov-type spaces }
\author{ Douadi Drihem and Zeghad Zouheyr }
\date{\today }
\maketitle

\begin{abstract}
In this paper we introduce Besov-type spaces with variable smoothness and
integrability. We show that these spaces are characterized by the $\varphi $%
-transforms in appropriate sequence spaces and we obtain atomic
decompositions for these spaces. Moreover the Sobolev embeddings for these
function spaces are obtained.


\end{abstract}

\section{Introduction}

Besov spaces of variable smoothness and integrability, $B_{p(\cdot ),q(\cdot
)}^{\alpha (\cdot )}$, initially appeared in the paper of Almeida and H\"{a}%
st\"{o} \cite{AH}. Several basic properties were established, such as the
Fourier analytical characterization and Sobolev embeddings. When $p,q,\alpha 
$ are constants they coincide with the usual function spaces $B_{p,q}^{s}$.
Later, \cite{D3} characterized these spaces by local means and established
the atomic characterization. Afterwards, Kempka and Vyb\'{\i}ral \cite{KV122}
characterized these spaces by the ball means of differences and also by
local means, see \cite{IN14} for the duality of $B_{p(\cdot ),q(\cdot
)}^{\alpha (\cdot )}$ spaces.\vskip5pt

Variable Besov-type spaces have been introduced in \cite{D5} and \cite{D6},
where their basic properties are given, such as the Sobolev type embeddings
and \ that under some conditions these spaces are just the variable Besov
spaces. For constant exponents, these spaces unify and generalize many
classical function spaces including Besov spaces, Besov-Morrey spaces (see,
for example, \cite[Corollary 3.3]{WYY}). Independently, D. Yang, C. Zhuo and
W. Yuan, \cite{YZW15} studied these function spaces where several properties
are obtained such as atomic decomposition and the boundedness of trace
operator. Also, Tyulenev \cite{Ty151}, \cite{Ty152} has studied a new
function spaces of variable smoothness. Triebel-Lizorkin spaces with
variable smoothness and integrability $F_{p(\cdot ),q(\cdot )}^{\alpha
(\cdot )}$\ were introduced in \cite{DHR}. They proved a discretization by
the so called $\varphi $-transform. Also atomic and molecular decomposition
of these function spaces are obtained and used it to derive trace results.
Subsequently, Vyb\.{\i}ral \cite{V09} established Sobolev-Jawerth embeddings
of these spaces.

The motivation to study such function spaces comes from applications to
other fields of applied mathematics, such that fluid dynamics and image
processing, see \cite{Ru00}.

The main aim of this paper is to present another Besov-type spaces with
variable smoothness and integrability which covers Besov-type spaces with
fixed exponents. We then establish their $\varphi $-transform
characterization in the sense of Frazier and Jawerth. We also characterize
these spaces by smooth atoms and give some basic properties and Sobolev-type
embeddings.\vskip5pt

The paper is organized as follows. First we give some preliminaries where we
fix some notation and recall some basics facts on function spaces with
variable integrability\ and we give some key technical lemmas needed in the
proofs of the main statements. For making the presentation clearer, we give
the proof of some lemmas later in Section 6. We then define the Besov-type
spaces $\mathfrak{B}_{p(\cdot ),q(\cdot )}^{\alpha (\cdot ),\tau (\cdot )}$.
In this section several basic properties such as the $\varphi $-transform
characterization are obtained. In Section 4 we prove elementary embeddings
between these functions spaces, as well as Sobolev embeddings. In Section 5,
we give the atomic decomposition of $\mathfrak{B}_{p(\cdot ),q(\cdot
)}^{\alpha (\cdot ),\tau (\cdot )}$ spaces.

\section{Preliminaries}

As usual, we denote by $\mathbb{R}^{n}$ the $n$-dimensional real Euclidean
space, $\mathbb{N}$ the collection of all natural numbers and $\mathbb{N}%
_{0}=\mathbb{N}\cup \{0\}$. The letter $\mathbb{Z}$ stands for the set of
all integer numbers.\ The expression $f\lesssim g$ means that $f\leq c\,g$
for some independent constant $c$ (and non-negative functions $f$ and $g$),
and $f\approx g$ means $f\lesssim g\lesssim f$. As usual for any $x\in 
\mathbb{R}$, $[x]$ stands for the largest integer smaller than or equal to $%
x $.\vskip5pt

By supp $f$ we denote the support of the function $f$, i.e., the closure of
its non-zero set. If $E\subset {\mathbb{R}^{n}}$ is a measurable set, then $%
|E|$ stands for the (Lebesgue) measure of $E$ and $\chi _{E}$ denotes its
characteristic function.\vskip5pt

The Hardy-Littlewood maximal operator $\mathcal{M}$ is defined on $L_{%
\mathrm{loc}}^{1}(\mathbb{R}^{n})$ by%
\begin{equation*}
\mathcal{M}f(x):=\sup_{r>0}\frac{1}{\left\vert B(x,r)\right\vert }%
\int_{B(x,r)}\left\vert f(y)\right\vert dy
\end{equation*}%
and 
\begin{equation*}
M_{B}f:=\frac{1}{\left\vert B\right\vert }\int_{B}\left\vert f(y)\right\vert
dy.
\end{equation*}%
The symbol $\mathcal{S}(\mathbb{R}^{n})$ is used in place of the set of all
Schwartz functions on $\mathbb{R}^{n}$. We denote by $\mathcal{S}^{\prime }(%
\mathbb{R}^{n})$ the dual space of all tempered distributions on $\mathbb{R}%
^{n}$. The Fourier transform of a tempered distribution $f$ is denoted by $%
\mathcal{F}f$ while its inverse transform is denoted by $\mathcal{F}^{-1}f$.%
\vskip5pt

For $v\in \mathbb{Z}$ and $m=(m_{1},...,m_{n})\in \mathbb{Z}^{n}$, let $%
Q_{v,m}$ be the dyadic cube in $\mathbb{R}^{n}$, $Q_{v,m}=%
\{(x_{1},...,x_{n}):m_{i}\leq 2^{v}x_{i}<m_{i}+1,i=1,2,...,n\}$. For the
collection of all such cubes we use 
\begin{equation*}
\mathcal{Q}:=\{Q_{v,m}:v\in \mathbb{Z},m\in \mathbb{Z}^{n}\}.
\end{equation*}%
For each cube $Q$, we denote its center by $c_{Q}$, its lower left-corner by 
$x_{Q_{v,m}}=2^{-v}m$ of $Q=Q_{v,m}$ and its side length by $l(Q)$. For $r>0$%
, we denote by $rQ$ the cube concentric with $Q$ having the side length $%
rl(Q)$. Furthermore, we put $v_{Q}=-\log _{2}l(Q)$ and $v_{Q}^{+}=\max
(v_{Q},0)$.\vskip5pt

For $v\in \mathbb{Z}$, $\varphi \in \mathcal{S}(\mathbb{R}^{n})$ and $x\in 
\mathbb{R}^{n}$, we set $\widetilde{\varphi }(x):=\overline{\varphi (-x)}$, $%
\varphi _{v}(x):=2^{vn}\varphi (2^{v}x)$, and%
\begin{equation*}
\varphi _{v,m}(x):=2^{vn/2}\varphi (2^{v}x-m)=|Q_{v,m}|^{1/2}\varphi
_{v}(x-x_{Q_{v,m}})\quad \text{if\quad }Q=Q_{v,m}.
\end{equation*}

By $c$ we denote generic positive constants, which may have different values
at different occurrences. Although the exact values of the constants are
usually irrelevant for our purposes, sometimes we emphasize their dependence
on certain parameters (e.g. $c(p)$ means that $c$ depends on $p$, etc.).
Further notation will be properly introduced whenever needed.

The variable exponents that we consider are always measurable functions $p$
on $\mathbb{R}^{n}$ with range in $[c,\infty \lbrack $ for some $c>0$. We
denote the set of such functions by $\mathcal{P}_{0}$. The subset of
variable exponents with range $[1,\infty \lbrack $ is denoted by $\mathcal{P}
$. We use the standard notation $p^{-}:=\underset{x\in \mathbb{R}^{n}}{\text{%
ess-inf}}$ $p(x)$ and $p^{+}:=\underset{x\in \mathbb{R}^{n}}{\text{ess-sup }}%
p(x)$.

The variable exponent modular is defined by 
\begin{equation*}
\varrho _{p(\cdot )}(f):=\int_{\mathbb{R}^{n}}\varrho _{p(x)}(\left\vert
f(x)\right\vert )dx,
\end{equation*}%
where $\varrho _{p}(t)=t^{p}$. The variable exponent Lebesgue space $%
L^{p(\cdot )}$\ consists of measurable functions $f$ on $\mathbb{R}^{n}$
such that $\varrho _{p(\cdot )}(\lambda f)<\infty $ for some $\lambda >0$.
We define the Luxemburg (quasi)-norm on this space by the formula 
\begin{equation*}
\big\|f\big\|_{p(\cdot )}:=\inf \Big\{\lambda >0:\varrho _{p(\cdot )}\Big(%
\frac{f}{\lambda }\Big)\leq 1\Big\}.
\end{equation*}%
A useful property is that $\left\Vert f\right\Vert _{p(\cdot )}\leq 1$ if
and only if $\varrho _{p(\cdot )}(f)\leq 1$, see \cite{DHHR}, Lemma 3.2.4.

Let $p,q\in \mathcal{P}_{0}$. The mixed Lebesgue-sequence space $\ell
^{q(\cdot )}(L^{p(\cdot )})$ is defined on sequences of $L^{p(\cdot )}$%
-functions by the semi-modular%
\begin{equation*}
\varrho _{\ell ^{q(\cdot )}(L^{p\left( \cdot \right)
})}((f_{v})_{v}):=\sum_{v}\inf \Big\{\lambda _{v}>0:\varrho _{p(\cdot )}\Big(%
\frac{f_{v}}{\lambda _{v}^{\frac{1}{q(\cdot )}}}\Big)\leq 1\Big\}.
\end{equation*}%
The (quasi)-norm is defined from this as usual:%
\begin{equation}
\left\Vert \left( f_{v}\right) _{v}\right\Vert _{\ell ^{q(\cdot
)}(L^{p\left( \cdot \right) })}:=\inf \Big\{\mu >0:\varrho _{\ell ^{q(\cdot
)}(L^{p(\cdot )})}\Big(\frac{1}{\mu }(f_{v})_{v}\Big)\leq 1\Big\}.
\label{mixed-norm}
\end{equation}%
If $q^{+}<\infty $, then we can replace \eqref{mixed-norm} by the simpler
expression 
\begin{equation*}
\varrho _{\ell ^{q(\cdot )}(L^{p(\cdot )})}((f_{v})_{v}):=\sum\limits_{v}%
\big\||f_{v}|^{q(\cdot )}\big\|_{\frac{p(\cdot )}{q(\cdot )}}.
\end{equation*}%
Furthermore, if $p$ and $q$ are constants, then $\ell ^{q(\cdot
)}(L^{p(\cdot )})=\ell ^{q}(L^{p})$. The case $p:=\infty $ can be included
by replacing the last semi-modular by 
\begin{equation*}
\varrho _{\ell ^{q(\cdot )}(L^{\infty })}((f_{v})_{v}):=\sum\limits_{v}\big\|%
\left\vert f_{v}\right\vert ^{q(\cdot )}\big\|_{\infty }.
\end{equation*}%
It is known, cf. \cite[Theorem 3.6]{AH} and \cite[Theorem 1]{KV121}, that $%
\ell ^{q(\cdot )}(L^{p(\cdot )})$ is a norm if $q(\cdot )\geq 1$ is constant
almost everywhere (a.e.) on $\mathbb{R}^{n}$ and $p(\cdot )\geq 1$, or if $%
\frac{1}{p(x)}+\frac{1}{q(x)}\leq 1$ a.e. on $\mathbb{R}^{n}$, or if $1\leq
q(x)\leq p(x)\leq \infty $ a.e. on $\mathbb{R}^{n}$.

We say that $g:\mathbb{R}^{n}\rightarrow \mathbb{R}$ is \textit{locally }log%
\textit{-H\"{o}lder continuous}, abbreviated $g\in C_{\text{loc}}^{\log }$,
if there exists $c_{\log }(g)>0$ such that%
\begin{equation}
\left\vert g(x)-g(y)\right\vert \leq \frac{c_{\log }(g)}{\log (e+\frac{1}{%
\left\vert x-y\right\vert })}  \label{lo-log-Holder}
\end{equation}%
for all $x,y\in \mathbb{R}^{n}$. We say that $g$ satisfies the log\textit{-H%
\"{o}lder decay condition}, if there exists $g_{\infty }\in \mathbb{R}$ and
a constant $c_{\log }>0$ such that%
\begin{equation*}
\left\vert g(x)-g_{\infty }\right\vert \leq \frac{c_{\log }}{\log
(e+\left\vert x\right\vert )}
\end{equation*}%
for all $x\in \mathbb{R}^{n}$. We say that $g$ is \textit{globally}-log%
\textit{-H\"{o}lder continuous}, abbreviated $g\in C^{\log }$, if it is%
\textit{\ }locally log-H\"{o}lder continuous and satisfies the log-H\"{o}%
lder decay\textit{\ }condition.\textit{\ }The constants $c_{\log }(g)$ and $%
c_{\log }$ are called the \textit{locally }log\textit{-H\"{o}lder constant }%
and the log\textit{-H\"{o}lder decay constant}, respectively\textit{.} We
note that all functions $g\in C_{\text{loc}}^{\log }$ always belong to $%
L^{\infty }$.\vskip5pt

We define the following class of variable exponents%
\begin{equation*}
\mathcal{P}^{\mathrm{log}}:=\Big\{p\in \mathcal{P}:\frac{1}{p}\in C^{\log }%
\Big\},
\end{equation*}%
were introduced in \cite[Section \ 2]{DHHMS}. We define $\frac{1}{p_{\infty }%
}:=\lim_{|x|\rightarrow \infty }\frac{1}{p(x)}$\ and we use the convention $%
\frac{1}{\infty }=0$. Note that although $\frac{1}{p}$ is bounded, the
variable exponent $p$ itself can be unbounded. It was shown in \cite{DHHR}%
\textrm{, }Theorem 4.3.8 that $\mathcal{M}:L^{p(\cdot )}\rightarrow
L^{p(\cdot )}$ is bounded if $p\in \mathcal{P}^{\mathrm{log}}$ and $p^{-}>1$%
, see also \cite{DHHMS}, Theorem 1.2.\ Also if $p\in \mathcal{P}^{\mathrm{log%
}}$, then the convolution with a radially decreasing $L^{1}$-function is
bounded on $L^{p(\cdot )}$: 
\begin{equation*}
\big\|\varphi \ast f\big\|_{{p(\cdot )}}\leq c\big\|\varphi \big\|_{{1}}%
\big\|f\big\|_{{p(\cdot )}}.
\end{equation*}%
We also refer to the papers \cite{CFMP} and \cite{Di}\textrm{,} where
various results on maximal function in variable Lebesgue spaces were
obtained.\vskip5pt

It is known that for $p\in \mathcal{P}^{\mathrm{log}}$ we have%
\begin{equation}
\big\|\chi _{B}\big\|_{{p(\cdot )}}\big\|\chi _{B}\big\|_{{p}^{\prime }{%
(\cdot )}}\approx |B|.  \label{DHHR}
\end{equation}%
with constants only depending on the $\log $-H\"{o}lder constant of $p$
(see, for example, \cite[Section 4.5]{DHHR}). Here ${p}^{\prime }$ denotes
the conjugate exponent of $p$ given by $\frac{1}{{p(\cdot )}}+\frac{1}{{p}%
^{\prime }{(\cdot )}}=1$. \vskip5pt

Recall that $\eta _{v,m}(x):=2^{nv}(1+2^{v}\left\vert x\right\vert )^{-m}$,
for any $x\in \mathbb{R}^{n}$, $v\in \mathbb{N}_{0}$ and $m>0$. Note that $%
\eta _{v,m}\in L^{1}$ when $m>n$ and that $\big\|\eta _{v,m}\big\|_{1}=c_{m}$
is independent of $v$, where this type of function was introduced in \cite%
{HN07} and \cite{DHHR}.

\subsection{Some technical lemmas}

In this subsection we present some results which are useful for us. The
following lemma is from \cite[Lemma 19]{KV122}, see also \cite[Lemma 6.1]%
{DHR}.

\begin{lemma}
\label{DHR-lemma}Let $\alpha \in C_{\mathrm{loc}}^{\log }$ and let $R\geq
c_{\log }(\alpha )$, where $c_{\log }(\alpha )$ is the constant from %
\eqref{lo-log-Holder} for $\alpha $. Then%
\begin{equation*}
2^{v\alpha (x)}\eta _{v,m+R}(x-y)\leq c\text{ }2^{v\alpha (y)}\eta
_{v,m}(x-y)
\end{equation*}%
with $c>0$ independent of $x,y\in \mathbb{R}^{n}$ and $v,m\in \mathbb{N}%
_{0}. $
\end{lemma}

The previous lemma allows us to treat the variable smoothness in many cases
as if it were not variable at all, namely we can move the term inside the
convolution as follows:%
\begin{equation*}
2^{v\alpha (x)}\eta _{v,m+R}\ast f(x)\leq c\text{ }\eta _{v,m}\ast
(2^{v\alpha (\cdot )}f)(x),\quad x\in \mathbb{R}^{n},
\end{equation*}%
where $c>0$ is independent of $v$ and $m$.

\begin{lemma}
\label{r-trick}Let $r,R,N>0$, $m>n$ and $\theta ,\omega \in \mathcal{S}%
\left( \mathbb{R}^{n}\right) $ with $\mathrm{supp}\mathcal{F}\omega \subset 
\overline{B(0,1)}$. Then there exists $c=c(r,m,n)>0$ such that for all $g\in 
\mathcal{S}^{\prime }\left( \mathbb{R}^{n}\right) $, we have%
\begin{equation}
\left\vert \theta _{R}\ast \omega _{N}\ast g\left( x\right) \right\vert \leq
c\text{ }A(\eta _{N,m}\ast \left\vert \omega _{N}\ast g\right\vert
^{r}(x))^{1/r},\quad x\in \mathbb{R}^{n},  \label{r-trick-est}
\end{equation}%
where $\theta _{R}=R^{n}\theta (R\cdot )$, $\omega _{N}=N^{n}\omega (N\cdot
) $, $\eta _{N,m}:=N^{n}(1+N\left\vert \cdot \right\vert )^{-m}$ and 
\begin{equation*}
A=\max \left( 1,\left( NR^{-1}\right) ^{m}\right) .
\end{equation*}
\end{lemma}

The proof of this lemma is given in \cite[Lemma 2.2]{D6}.

We will make use of the following statement, see \cite{DHHMS}, Lemma 3.3.

\begin{lemma}
\label{DHHR-estimate}Let $p\in \mathcal{P}^{\mathrm{log}}$. Then for every $%
m>0$ there exists $\beta \in \left( 0,1\right) $ only depending on $m$ and $%
c_{\mathrm{log}}\left( p\right) $ such that%
\begin{eqnarray*}
&&\Big(\frac{\beta }{\left\vert Q\right\vert }\int_{Q}\left\vert
f(y)\right\vert dy\Big)^{p\left( x\right) } \\
&\leq &\frac{1}{\left\vert Q\right\vert }\int_{Q}\left\vert f(y)\right\vert
^{p\left( y\right) }dy \\
&&+\min \left( \left\vert Q\right\vert ^{m},1\right) \Big(\frac{1}{%
\left\vert Q\right\vert }\int_{Q}\left( \left( e+\left\vert x\right\vert
\right) ^{-m}+\left( e+\left\vert y\right\vert \right) ^{-m}\right) dy\Big),
\end{eqnarray*}%
for every cube $\mathrm{(}$or ball$\mathrm{)}$ $Q\subset \mathbb{R}^{n}$,
all $x\in Q\subset \mathbb{R}^{n}$and all $f\in L^{p\left( \cdot \right)
}+L^{\infty }$\ such that $\big\|f\big\|_{L^{p\left( \cdot \right)
}+L^{\infty }}\leq 1$.
\end{lemma}

Notice that in the proof of this lemma we need only that 
\begin{equation*}
\int_{Q}\left\vert f(y)\right\vert ^{p\left( y\right) }dy\leq 1
\end{equation*}%
and/or $\left\Vert f\right\Vert _{\infty }\leq 1$. We set%
\begin{equation*}
\big\|\left( f_{v}\right) _{v}\big\|_{\ell ^{q(\cdot )}(L_{p(\cdot )}^{\tau
(\cdot )})}:=\sup_{P\in \mathcal{Q}}\Big\|\Big(\frac{f_{v}}{|P|^{\tau (\cdot
)}}\chi _{P}\Big)_{v\geq v_{P}^{+}}\Big\|_{\ell ^{q(\cdot )}(L^{p(\cdot )})},
\end{equation*}%
where, $v_{P}=-\log _{2}l(P)$ and $v_{P}^{+}=\max (v_{P},0)$.

The following lemma is the $\ell ^{q(\cdot )}(L_{p(\cdot )}^{\tau (\cdot )})$%
-version of Lemma 4.7 from Almeida and H\"{a}st\"{o} \cite{AH} (we use it,
since the maximal operator is in general not bounded on $\ell ^{q(\cdot
)}(L^{p(\cdot )})$, see \cite[Example 4.1]{AH}).

\begin{lemma}
\label{Alm-Hastolemma1}Let $\mathbb{\tau }\in C_{\mathrm{loc}}^{\log }$, $%
\tau ^{-}>0$, $p\in \mathcal{P}^{\log }$, $q\in \mathcal{P}_{0}^{\log }$
with $0<q^{-}\leq q^{+}<\infty $ and $\tau ^{+}<\left( \tau p\right) ^{-}$.
For any $m$ large enough, there exists $c>0$ such that%
\begin{equation*}
\big\|(\eta _{v,m}\ast f_{v})_{v}\big\|_{\ell ^{q(\cdot )}(L_{p(\cdot
)}^{\tau (\cdot )})}\leq c\big\|(f_{v})_{v}\big\|_{\ell ^{q(\cdot
)}(L_{p(\cdot )}^{\tau (\cdot )})}
\end{equation*}%
for any $(f_{v})_{v}\in \ell ^{q(\cdot )}(L_{p(\cdot )}^{\tau (\cdot )})$.
\end{lemma}

The proof of this lemma is postponed to the Appendix.

Let $\widetilde{L_{\tau (\cdot )}^{p(\cdot )}}$ be the collection of
functions $f\in L_{\text{loc}}^{p(\cdot )}(\mathbb{R}^{n})$ such that%
\begin{equation*}
\big\|f\big\|_{\widetilde{L_{\tau (\cdot )}^{p(\cdot )}}}:=\sup \Big\|\frac{%
f\chi _{P}}{|P|^{\tau (\cdot )}}\Big\|_{p(\cdot )}<\infty ,\quad p\in 
\mathcal{P}_{0},\quad \tau :\mathbb{R}^{n}\rightarrow \mathbb{R}^{+},
\end{equation*}%
where the supremum is taken over all dyadic cubes $P$ with $|P|\geq 1$.
Notice that 
\begin{equation}
\left\Vert f\right\Vert _{\widetilde{L_{\tau (\cdot )}^{p(\cdot )}}}\leq
1\Leftrightarrow \sup_{P\in \mathcal{Q},|P|\geq 1}\Big\|\Big|\frac{f}{%
|P|^{\tau (\cdot )}}\Big|^{q(\cdot )}\chi _{P}\Big\|_{p(\cdot )/q(\cdot
)}\leq 1.  \label{mod-est}
\end{equation}

Recall that\ $\theta _{v}=2^{vn}\theta \left( 2^{v}\cdot \right) ,v\in 
\mathbb{Z}$.

\begin{lemma}
\label{key-estimate1}Let $v\in \mathbb{Z}$, $\mathbb{\tau }\in C_{\mathrm{loc%
}}^{\log }$, $\tau ^{-}>0$, $p\in \mathcal{P}_{0}^{\log }$ and $\theta
,\omega \in \mathcal{S}(\mathbb{R}^{n})$ with $\mathrm{supp}\mathcal{F}%
\omega \subset \overline{B(0,1)}$. For any $f\in \mathcal{S}^{\prime }(%
\mathbb{R}^{n})$ and any dyadic cube $P$ with $|P|\geq 1$, we have%
\begin{equation*}
\Big\|\frac{\theta _{v}\ast \omega _{v}\ast f}{|P|^{\tau (\cdot )}}\chi _{P}%
\Big\|_{p(\cdot )}\leq c\left\Vert \omega _{v}\ast f\right\Vert _{\widetilde{%
L_{\tau (\cdot )}^{p(\cdot )}}},
\end{equation*}%
such that the right-hand side is finite, where $c>0$ is independent of $v$
and $l(P).$
\end{lemma}

We will present the proof in Appendix.

\begin{lemma}
\label{Key-lemma}Let $\alpha ,\mathbb{\tau }\in C_{\mathrm{loc}}^{\log }$, $%
\mathbb{\tau }^{-}\geq 0$ and $p,q\in \mathcal{P}_{0}^{\log }$ with $%
0<q^{-}\leq q^{+}<\infty $. Let $(f_{k})_{k\in \mathbb{N}_{0}}$ be a
sequence of measurable functions on $\mathbb{R}^{n}$. For all $v\in \mathbb{N%
}_{0}$ and $x\in \mathbb{R}^{n}$, let 
\begin{equation*}
g_{v}(x)=\sum_{k=0}^{\infty }2^{-|k-v|\delta }f_{k}(x).
\end{equation*}%
Then there exists a positive constant $c$, independent of $(f_{k})_{k\in 
\mathbb{N}_{0}}$ such that%
\begin{equation*}
\big\|(g_{v})_{v}\big\|_{\ell ^{q(\cdot )}(L_{p(\cdot )}^{\tau (\cdot
)})}\leq c\big\|(f_{v})_{v}\big\|_{\ell ^{q(\cdot )}(L_{p(\cdot )}^{\tau
(\cdot )})},\quad \delta >0.
\end{equation*}
\end{lemma}

The proof of Lemma \ref{Key-lemma} can be obtained by the same arguments
used in \cite[Lemma 2.10]{D6}.

\section{The spaces\textbf{\ }$\mathfrak{B}_{p(\cdot ),q(\cdot )}^{\protect%
\alpha (\cdot ),\protect\tau (\cdot )}$}

In this section we\ present the Fourier analytical definition of Besov-type
spaces of variable smoothness and integrability\ and we prove their basic
properties in analogy to the Besov-type spaces with fixed exponents. Select
a pair of Schwartz functions $\Phi $ and $\varphi $ such that%
\begin{equation}
\text{supp}\mathcal{F}\Phi \subset \overline{B(0,2)}\text{\quad and\quad }|%
\mathcal{F}\Phi (\xi )|\geq c\text{\quad if\quad }|\xi |\leq \frac{5}{3}
\label{Ass1}
\end{equation}%
and 
\begin{equation}
\text{supp}\mathcal{F}\varphi \subset \overline{B(0,2)}\backslash B(0,1/2)%
\text{\quad and\quad }|\mathcal{F}\varphi (\xi )|\geq c\text{\quad if\quad }%
\frac{3}{5}\leq |\xi |\leq \frac{5}{3},  \label{Ass2}
\end{equation}%
where $c>0$. We put $\varphi _{v}=2^{vn}\varphi (2^{v}\cdot ),v\in \mathbb{N}%
.$

\begin{definition}
\label{B-F-def}Let $\alpha :\mathbb{R}^{n}\rightarrow \mathbb{R}$, $\tau :%
\mathbb{R}^{n}\rightarrow \mathbb{R}^{+}$ and $p,q\in \mathcal{P}_{0}$. Let $%
\Phi $ and $\varphi $ satisfy \eqref{Ass1} and \eqref{Ass2}, respectively.
The Besov-type space $\mathfrak{B}_{p(\cdot ),q(\cdot )}^{\alpha (\cdot
),\tau (\cdot )}$\ is the collection of all $f\in \mathcal{S}^{\prime }(%
\mathbb{R}^{n})$\ such that 
\begin{equation}
\left\Vert f\right\Vert _{\mathfrak{B}_{p(\cdot ),q(\cdot )}^{\alpha (\cdot
),\tau (\cdot )}}:=\sup_{P\in \mathcal{Q}}\Big\|\Big(\frac{2^{v\alpha \left(
\cdot \right) }\varphi _{v}\ast f}{|P|^{\tau (\cdot )}}\chi _{P}\Big)_{v\geq
v_{P}^{+}}\Big\|_{\ell ^{q(\cdot )}(L^{p(\cdot )})}<\infty ,  \label{B-def}
\end{equation}%
where $\varphi _{0}$ is replaced by $\Phi $.
\end{definition}

Using the system $(\varphi _{v})_{v\in \mathbb{N}_{0}}$ we can define the
quasi-norm%
\begin{equation*}
\left\Vert f\right\Vert _{B_{p,q}^{\alpha ,\tau }}:=\sup_{P\in \mathcal{Q}}%
\frac{1}{\left\vert P\right\vert ^{\tau }}\Big(\sum\limits_{v=v_{P}^{+}}^{%
\infty }2^{v\alpha q}\left\Vert \left( \varphi _{v}\ast f\right) \chi
_{P}\right\Vert _{p}^{q}\Big)^{\frac{1}{q}}
\end{equation*}%
for constants $\alpha $ and $p,q\in (0,\infty ]$, with the usual
modification if $q=\infty $. The Besov-type space $B_{p,q}^{\alpha ,\tau }$
consist of all distributions $f\in \mathcal{S}^{\prime }(\mathbb{R}^{n})$
for which $\left\Vert f\right\Vert _{B_{p,q}^{\alpha ,\tau }}<\infty $. It
is well-known that these spaces do not depend on the choice of the system $%
(\varphi _{v})_{v\in \mathbb{N}_{0}}$ (up to equivalence of quasinorms).
Further details on the classical theory of these spaces can be found in \cite%
{D1}, \cite{YY13} and \cite{WYY}, see also \cite{D4} for recent
developments. Moreover, $B_{p,q}^{\alpha ,0}$ are just the classical Besov
spaces, see \cite{T1} for the theory of these function spaces.

One recognizes immediately that if $\alpha $, $\tau $, $p$ and $q$ are
constants, then 
\begin{equation*}
\mathfrak{B}_{p(\cdot ),q(\cdot )}^{\alpha (\cdot ),\tau (\cdot
)}=B_{p,q}^{\alpha ,\tau }.
\end{equation*}%
When, $q:=\infty $\ the Besov-type space $\mathfrak{B}_{p(\cdot ),\infty
}^{\alpha (\cdot ),\tau (\cdot )}$\ consist of all distributions $f\in 
\mathcal{S}^{\prime }(\mathbb{R}^{n})$\ such that 
\begin{equation*}
\sup_{P\in \mathcal{Q},v\geq v_{P}^{+}}\Big\|\frac{2^{v\alpha \left( \cdot
\right) }\varphi _{v}\ast f}{|P|^{\tau (\cdot )}}\chi _{P}\Big\|_{p(\cdot
)}<\infty .
\end{equation*}%
Let $B_{J}$ be any ball of $\mathbb{R}^{n}$ with radius $2^{-J}$, $J\in 
\mathbb{Z}$. In the definition of the spaces $\mathfrak{B}_{p(\cdot
),q(\cdot )}^{\alpha (\cdot ),\tau (\cdot )}$ if we replace the dyadic cubes 
$P$ by the balls $B_{J}$, then we obtain equivalent quasi-norms. From these
if we replace dyadic cubes $P$ in Definition \ref{B-F-def} by arbitrary
cubes $P$, we then obtain equivalent quasi-norms.

The Besov space of variable smoothness and integrability $B_{p(\cdot
),q(\cdot )}^{\alpha (\cdot )}$ is the collection of all $f\in \mathcal{S}%
^{\prime }(\mathbb{R}^{n})$\ such that 
\begin{equation*}
\left\Vert f\right\Vert _{B_{p(\cdot ),q(\cdot )}^{\alpha (\cdot
)}}:=\left\Vert \left( 2^{v\alpha \left( \cdot \right) }\varphi _{v}\ast
f\right) _{v\in \mathbb{N}_{0}}\right\Vert _{\ell ^{q(\cdot )}(L^{p(\cdot
)})}<\infty ,
\end{equation*}%
which introduced and investigated in \cite{AH}, see \cite{KV122}\ for
further results. Taking $\alpha \in \mathbb{R}$ and $q\in (0,\infty )$ as
constants we derive the spaces $B_{p(\cdot ),q}^{\alpha }$ studied by Xu in 
\cite{Xu08}. Obviously, 
\begin{equation*}
\mathfrak{B}_{p(\cdot ),q}^{\alpha ,0}=B_{p(\cdot ),q}^{\alpha }.
\end{equation*}%
We refer the reader to the recent paper \cite{YHSY} for further details,
historical remarks and more references on embeddings of Besov-type spaces
with fixed exponents. We mention that the variable Triebel-Lizorkin version
of our spaces introduced on this paper is given in \cite{D7}. Variable
Besov-Morrey spaces are given in \cite{AC19}, see \cite{C21} and \cite%
{Dachun18} for the variable 2-microlocal Besov-Triebel-Lizorkin-type spaces.

Sometimes it is of great service if one can restrict sup$_{P\in \mathcal{Q}}$
in the definition of $\mathfrak{B}_{p(\cdot ),q(\cdot )}^{\alpha (\cdot
),\tau (\cdot )}$ to a supremum taken with respect to dyadic cubes with side
length $\leq 1$. The next lemma can be obtained by an argument similar to
that used in the proof of \cite[Lemma 3.6]{D5}.

\begin{lemma}
\label{new-equinorm}Let $\alpha ,\mathbb{\tau }\in C_{\mathrm{loc}}^{\log }$%
, $\mathbb{\tau }^{-}\geq 0$ and $p,q\in \mathcal{P}_{0}^{\log }$ with $%
\left( \mathbb{\tau }p-1\right) ^{-}\geq 0$ and $0<q^{+}<\infty $. A
tempered distribution $f$ belongs to $\mathfrak{B}_{p(\cdot ),q(\cdot
)}^{\alpha (\cdot ),\tau (\cdot )}$ if and only if,%
\begin{equation*}
\left\Vert f\right\Vert _{\mathfrak{B}_{p(\cdot ),q(\cdot )}^{\alpha (\cdot
),\tau (\cdot )}}^{\#}:=\sup_{P\in \mathcal{Q},|P|\leq 1}\Big\|\Big(\frac{%
2^{v\alpha \left( \cdot \right) }\varphi _{v}\ast f}{|P|^{\tau (\cdot )}}%
\chi _{P}\Big)_{v\geq v_{P}}\Big\|_{\ell ^{q(\cdot )}(L^{p(\cdot )})}<\infty
.
\end{equation*}%
Furthermore, the quasi-norms $\left\Vert f\right\Vert _{\mathfrak{B}%
_{p(\cdot ),q(\cdot )}^{\alpha (\cdot ),\tau (\cdot )}}$ and $\left\Vert
f\right\Vert _{\mathfrak{B}_{p(\cdot ),q(\cdot )}^{\alpha (\cdot ),\tau
(\cdot )}}^{\#}$ are equivalent.
\end{lemma}

\begin{remark}
We like to point out that this result with fixed exponents is given in \cite[%
Lemma 2.2]{WYY}.
\end{remark}

The following conclusion implies under some suitable conditions the variable
Besov-type spaces $\mathfrak{B}_{p(\cdot ),q(\cdot )}^{\alpha (\cdot ),\tau
(\cdot )}$ are just the Besov spaces $B_{\infty ,\infty }^{\alpha (\cdot
)+n(\tau (\cdot )-1/p(\cdot ))}$, whose proof is similar to that of \cite[%
Theorem 3.8]{D5}, the details being omitted.

\begin{theorem}
Let $\alpha ,\mathbb{\tau }\in C_{\mathrm{loc}}^{\log }$, $\mathbb{\tau }%
^{-}\geq 0$ and $p,q\in \mathcal{P}_{0}^{\log }$ with $p^{+},q^{+}<\infty $.
If $\left( \mathbb{\tau }p-1\right) ^{-}>0$ or $\left( \mathbb{\tau }%
p-1\right) ^{-}\geq 0$ and $q:=\infty $, then%
\begin{equation*}
\mathfrak{B}_{p(\cdot ),q(\cdot )}^{\alpha (\cdot ),\tau (\cdot )}=B_{\infty
,\infty }^{\alpha (\cdot )+n(\tau (\cdot )-\frac{1}{p(\cdot )})}
\end{equation*}%
with equivalent quasi-norms.
\end{theorem}

\begin{remark}
\label{new-est}$\mathrm{From}$ $\mathrm{this}$ $\mathrm{theorem}$ $\mathrm{we%
}$ $\mathrm{obtain}$%
\begin{equation}
2^{v(\alpha (x)+n(\tau (x)-\frac{1}{p(x)})}|\varphi _{v}\ast f(x)|\leq c%
\big\|f\big\|_{\mathfrak{B}_{p(\cdot ),q(\cdot )}^{\alpha (\cdot ),\tau
(\cdot )}}  \label{emd}
\end{equation}%
$\mathrm{for}$ $\mathrm{any}$ $f\in \mathfrak{B}_{p(\cdot ),q(\cdot
)}^{\alpha (\cdot ),\tau (\cdot )}$, $x\in \mathbb{R}^{n}$, $\alpha ,\tau
\in C_{\mathrm{loc}}^{\log }$, $\mathbb{\tau }^{-}\geq 0$ $\mathrm{and}$ $%
p,q\in \mathcal{P}_{0}^{\log },$ $\mathrm{where}$ $c>0$ $\mathrm{is}$ $%
\mathrm{independent}$ $\mathrm{of}$ $v$ $\mathrm{and}$ $x\mathrm{.}$
\end{remark}

In the following theorem we have the possibility to define these spaces by
replacing $v\geq v_{P}^{+}$ by $v\in \mathbb{N}_{0}$ in Definition \ref%
{B-F-def}, where the main arguments used in its proof rely on \cite[Theorem
3.11]{D5}, so we omit the details and when $\tau :=0$, was obtained by
Sickel \cite{Si}.

\begin{theorem}
Let $\alpha ,\mathbb{\tau }\in C_{\mathrm{loc}}^{\log }$, $\mathbb{\tau }%
^{-}\geq 0$ and $p,q\in \mathcal{P}_{0}^{\log }$ with $p^{+},q^{+}<\infty $.
If $(\tau p-1)^{+}<0$ or\ $(\tau p-1)^{+}\leq 0$\ and $q:=\infty $, then%
\begin{equation*}
\big\|f\big\|_{\mathfrak{B}_{p(\cdot ),q(\cdot )}^{\alpha (\cdot ),\tau
(\cdot )}}^{\ast }=\sup_{P\in \mathcal{Q}}\Big\|\Big(\frac{2^{\alpha \left(
\cdot \right) }\varphi _{v}\ast f}{|P|^{\tau (\cdot )}}\chi _{P}\Big)_{v\in 
\mathbb{N}_{0}}\Big\|_{\ell ^{q(\cdot )}(L^{p(\cdot )})},
\end{equation*}%
is an equivalent quasi-norm in $\mathfrak{B}_{p(\cdot ),q(\cdot )}^{\alpha
(\cdot ),\tau (\cdot )}$.
\end{theorem}

Let $\Phi $ and $\varphi $ satisfy, respectively \eqref{Ass1} and %
\eqref{Ass2}. By \cite[pp. 130--131]{FJ90}, there exist \ functions $\Psi
\in \mathcal{S}(\mathbb{R}^{n})$ satisfying \eqref{Ass1} and $\psi \in 
\mathcal{S}(\mathbb{R}^{n})$ satisfying \eqref{Ass2} such that for all $\xi
\in \mathbb{R}^{n}$%
\begin{equation}
\mathcal{F}\widetilde{\Phi }(\xi )\mathcal{F}\Psi (\xi )+\sum_{j=1}^{\infty }%
\mathcal{F}\widetilde{\varphi }(2^{-j}\xi )\mathcal{F}\psi (2^{-j}\xi
)=1,\quad \xi \in \mathbb{R}^{n},  \label{Ass4}
\end{equation}%
where $\widetilde{\Phi }=\overline{\Phi (-\cdot )}$ and $\widetilde{\varphi }%
=\overline{\varphi (-\cdot )}$. Furthermore, we have the following identity
for all $f\in \mathcal{S}^{\prime }(\mathbb{R}^{n})$; see \cite[(12.4)]{FJ90}%
\begin{eqnarray*}
f &=&\Psi \ast \widetilde{\Phi }\ast f+\sum_{v=1}^{\infty }\psi _{v}\ast 
\widetilde{\varphi }_{v}\ast f \\
&=&\sum_{m\in \mathbb{Z}^{n}}\widetilde{\Phi }\ast f(m)\Psi (\cdot
-m)+\sum_{v=1}^{\infty }2^{-vn}\sum_{m\in \mathbb{Z}^{n}}\widetilde{\varphi }%
_{v}\ast f(2^{-v}m)\psi _{v}(\cdot -2^{-v}m).
\end{eqnarray*}%
Recall that the $\varphi $-transform $S_{\varphi }$ is defined by setting $%
(S_{\varphi })_{0,m}=\langle f,\Phi _{m}\rangle $ where $\Phi _{m}(x)=\Phi
(x-m)$ and $(S_{\varphi })_{v,m}=\langle f,\varphi _{v,m}\rangle $ where $%
\varphi _{v,m}(x)=2^{vn/2}\varphi (2^{v}x-m)$ and $v\in \mathbb{N}$. The
inverse $\varphi $-transform $T_{\psi }$ is defined by 
\begin{equation*}
T_{\psi }\lambda =\sum_{m\in \mathbb{Z}^{n}}\lambda _{0,m}\Psi
_{m}+\sum_{v=1}^{\infty }\sum_{m\in \mathbb{Z}^{n}}\lambda _{v,m}\psi _{v,m},
\end{equation*}%
where $\lambda =\{\lambda _{v,m}\in \mathbb{C}:v\in \mathbb{N}_{0},m\in 
\mathbb{Z}^{n}\}$, see \cite{FJ90}.

For any $\gamma \in \mathbb{Z}$, we put%
\begin{equation*}
\left\Vert f\right\Vert _{\mathfrak{B}_{p(\cdot ),q(\cdot )}^{\alpha (\cdot
),\tau (\cdot )}}^{\ast }:=\sup_{P\in \mathcal{Q}}\Big\|\Big(\frac{%
2^{v\alpha \left( \cdot \right) }\varphi _{v}\ast f}{|P|^{\tau (\cdot )}}%
\chi _{P}\Big)_{v\geq v_{P}^{+}-\gamma }\Big\|_{\ell ^{q(\cdot )}(L^{p(\cdot
)})}<\infty
\end{equation*}%
where $\varphi _{-\gamma }$ is replaced by $\Phi _{-\gamma }$.

\begin{lemma}
\label{new-equinorm3}Let $\alpha ,\mathbb{\tau }\in C_{\mathrm{loc}}^{\log }$%
, $\tau ^{-}>0$, $p,q\in \mathcal{P}_{0}^{\log }$ and $0<q^{+}<\infty $. The
quasi-norms $\big\|f\big\|_{\mathfrak{B}_{p(\cdot ),q(\cdot )}^{\alpha
(\cdot ),\tau (\cdot )}}^{\ast }$ and $\big\|f\big\|_{\mathfrak{B}_{p(\cdot
),q(\cdot )}^{\alpha (\cdot ),\tau (\cdot )}}$ are equivalent with
equivalent constants depending on $\gamma $.
\end{lemma}

\begin{proof}
The proof is a straightforward adaptation of \cite[Lemma 3.9]{D6} and \cite[%
Lemma 2.6]{WYY}. We will present the proof into two steps.

\textit{Step 1.} In this step we prove that 
\begin{equation*}
\big\|f\big\|_{\mathfrak{B}_{p(\cdot ),q(\cdot )}^{\alpha (\cdot ),\tau
(\cdot )}}^{\ast }\lesssim \big\|f\big\|_{\mathfrak{B}_{p(\cdot ),q(\cdot
)}^{\alpha (\cdot ),\tau (\cdot )}}.
\end{equation*}%
We need only to consider the case $\gamma >0.$ By the scaling argument, it
suffices to consider the case 
\begin{equation}
\big\|f\big\|_{\mathfrak{B}_{p(\cdot ),q(\cdot )}^{\alpha (\cdot ),\tau
(\cdot )}}=1  \label{modular}
\end{equation}%
and show that the modular of $f$ on the left-hand side is bounded. In
particular, we will show that%
\begin{equation*}
\sum_{v=v_{P}^{+}-\gamma }^{\infty }\Big\|\Big|\frac{2^{v\alpha \left( \cdot
\right) }\varphi _{v}\ast f}{|P|^{\tau \left( \cdot \right) }}\Big|^{q\left(
\cdot \right) }\chi _{P}\Big\|_{\frac{{p(\cdot )}}{q\left( \cdot \right) }%
}\leq c
\end{equation*}%
for any dyadic cube $P$. As in \cite[Lemma 2.6]{WYY}, it suffices to prove
that for all dyadic cube $P$ with $l(P)\geq 1$,%
\begin{equation*}
I_{P}=\sum_{v=-\gamma }^{0}\Big\|\Big|\frac{2^{v\alpha \left( \cdot \right)
}\varphi _{v}\ast f}{|P|^{\tau \left( \cdot \right) }}\Big|^{q\left( \cdot
\right) }\chi _{P}\Big\|_{\frac{{p(\cdot )}}{q\left( \cdot \right) }}\leq c
\end{equation*}%
and for all dyadic cube $P$ with $l(P)<1,$%
\begin{equation*}
J_{P}=\sum_{v=v_{p}-\gamma }^{v_{p}-1}\Big\|\Big|\frac{2^{v\alpha \left(
\cdot \right) }\varphi _{v}\ast f}{|P|^{\tau \left( \cdot \right) }}\Big|%
^{q\left( \cdot \right) }\chi _{P}\Big\|_{\frac{{p(\cdot )}}{q\left( \cdot
\right) }}\leq c.
\end{equation*}%
The estimate of $I_{P},$ clearly follows from the inequality 
\begin{equation*}
\Big\|\Big|\frac{\varphi _{v}\ast f}{|P|^{\tau \left( \cdot \right) }}\Big|%
^{q\left( \cdot \right) }\chi _{P}\Big\|_{\frac{{p(\cdot )}}{q\left( \cdot
\right) }}\leq c
\end{equation*}%
for any $v=-\gamma ,...,0$ and any dyadic cube $P$ with $l(P)\geq 1$. This
claim can be reformulated as showing that%
\begin{equation}
\Big\|\frac{\varphi _{v}\ast f}{|P|^{\tau \left( \cdot \right) }}\chi _{P}%
\Big\|_{p\left( \cdot \right) }\leq c\text{.}  \label{pr1}
\end{equation}%
From $\mathrm{\eqref{Ass1}}$ and $\mathrm{\eqref{Ass2}}$, we find $\omega
_{v}\in \mathcal{S}(%
\mathbb{R}
^{n})$, $v=-\gamma ,...,-1$ and $\eta _{1},\eta _{2}\in \mathcal{S}(%
\mathbb{R}
^{n})$ such that%
\begin{equation*}
\varphi _{v}=\omega _{v}\ast \Phi ,\text{\quad }v=-\gamma ,...,-1\text{\quad
and\quad }\varphi =\varphi _{0}=\eta _{1}\ast \Phi +\eta _{2}\ast \varphi
_{1}.
\end{equation*}%
Therefore 
\begin{equation*}
\varphi _{v}\ast f=\omega _{v}\ast \Phi \ast f\quad \text{for}\quad
v=-\gamma ,...,-1
\end{equation*}%
and 
\begin{equation*}
\varphi _{0}\ast f=\eta _{1}\ast \Phi \ast f+\eta _{2}\ast \varphi _{1}\ast
f.
\end{equation*}%
Using Lemma \ref{key-estimate1}, $\mathrm{\eqref{mod-est}}$ and $\mathrm{%
\eqref{modular}}$ to estimate the left-hand side of (\ref{pr1}) by%
\begin{equation*}
C\big\|\Phi \ast f\big\|_{\widetilde{L_{\tau \left( \cdot \right) }^{p\left(
\cdot \right) }}}+C\big\|\varphi _{1}\ast f\big\|_{\widetilde{L_{\tau \left(
\cdot \right) }^{p\left( \cdot \right) }}}\leq c.
\end{equation*}%
To estimate $J_{P}$, denote by $P(\gamma )$ the dyadic cube containing $P$
with $l(P(\gamma ))=2^{\gamma }l(P).$ If $v_{P}\geq \gamma +1$, applying the
fact that $v_{P(\gamma )}=v_{P}-\gamma ,$ and $P\subset P(\gamma ),$ we then
have%
\begin{equation*}
J_{P}\leq \sum_{v=v_{P\left( \gamma \right) }}^{v_{p}-1}\Big\|\Big|\frac{%
2^{v\alpha \left( \cdot \right) }\varphi _{v}\ast f}{|P\left( \gamma \right)
|^{\tau \left( \cdot \right) }}\Big|^{q\left( \cdot \right) }\chi _{P\left(
\gamma \right) }\Big\|_{\frac{{p(\cdot )}}{q\left( \cdot \right) }}\leq c.
\end{equation*}%
If $1\leq v_{P}\leq \gamma ,$ we write 
\begin{eqnarray*}
J_{P} &=&\sum_{v=v_{P}-\gamma }^{-1}...+\sum_{v=0}^{v_{P}-1}... \\
&=&J_{P}^{1}+J_{P}^{2}.
\end{eqnarray*}%
Let $P(2^{v_{P}})$ the dyadic cube containing $P$ with $%
l(P(2^{v_{P}}))=2^{v_{P}}l(P)=1.$ By the fact that 
\begin{equation*}
\frac{|P(2^{v_{P}})|^{\tau \left( \cdot \right) }}{|P|^{\tau \left( \cdot
\right) }}\lesssim 2^{nv_{P}\tau ^{+}}\lesssim c(\gamma ),
\end{equation*}%
we have%
\begin{equation*}
J_{P}^{2}\lesssim \sum_{v=v_{P(2^{v_{P}})}}^{v_{p}-1}\Big\|\Big|\frac{%
2^{v\alpha \left( \cdot \right) }\varphi _{v}\ast f}{|P(2^{v_{P}})|^{\tau
\left( \cdot \right) }}\Big|^{q\left( \cdot \right) }\chi _{P(2^{v_{P}})}%
\Big\|_{\frac{{p(\cdot )}}{q\left( \cdot \right) }}\leq c.
\end{equation*}%
By the arguments similar to that uesd in the estimate for $I_{P}$ , we
obtain $J_{P}^{1}\leq c.$

\textit{Step 2.} We will prove that%
\begin{equation*}
\big\|f\big\|_{\mathfrak{B}_{p(\cdot ),q(\cdot )}^{\alpha (\cdot ),\tau
(\cdot )}}\lesssim \big\|f\big\|_{\mathfrak{B}_{p(\cdot ),q(\cdot )}^{\alpha
(\cdot ),\tau (\cdot )}}^{\ast }.
\end{equation*}%
It suffices to show that%
\begin{equation*}
\Big\|\Big|\frac{\Phi \ast f}{|P|^{\tau \left( \cdot \right) }}\Big|%
^{q\left( \cdot \right) }\chi _{P}\Big\|_{\frac{{p(\cdot )}}{q\left( \cdot
\right) }}\leq c
\end{equation*}%
for all $P\in \mathcal{Q}$ with $l(P)\geq 1$ and all $f\in \mathcal{B}%
_{p\left( \cdot \right) ,q\left( \cdot \right) }^{\alpha \left( \cdot
\right) ,\tau \left( \cdot \right) }$ with 
\begin{equation*}
\big\|f\big\|_{\mathcal{B}_{p\left( \cdot \right) ,q\left( \cdot \right)
}^{\alpha \left( \cdot \right) ,\tau \left( \cdot \right) }}^{\ast }\leq 1.
\end{equation*}%
This claim can be reformulated as showing that 
\begin{equation*}
\big\|\frac{\Phi \ast f}{|P|^{\tau \left( \cdot \right) }}\chi _{P}\big\|%
_{p\left( \cdot \right) }\leq c.
\end{equation*}%
There exist $\varrho _{v}\in \mathcal{S}(%
\mathbb{R}
^{n})$, $v=-\gamma ,...,1$, such that 
\begin{equation*}
\Phi \ast f=\varrho _{-\gamma }\ast \Phi _{-\gamma }\ast f+\sum_{v=1-\gamma
}^{1}\varrho _{v}\ast \varphi _{v}\ast f,
\end{equation*}%
see \cite[p. 130]{FJ90}. Using Lemma \ref{key-estimate1} we get%
\begin{equation*}
\big\|\varrho _{-\gamma }\ast \Phi _{-\gamma }\ast f\big\|_{\widetilde{%
L_{\tau \left( \cdot \right) }^{p\left( \cdot \right) }}}\lesssim \big\|\Phi
_{-\gamma }\ast f\big\|_{\widetilde{L_{\tau \left( \cdot \right) }^{p\left(
\cdot \right) }}}\leq c,
\end{equation*}%
and%
\begin{equation*}
\big\|\varrho _{v}\ast \varphi _{v}\ast f\big\|_{\widetilde{L_{\tau \left(
\cdot \right) }^{p\left( \cdot \right) }}}\lesssim \big\|\varphi _{v}\ast f%
\big\|_{\widetilde{L_{\tau \left( \cdot \right) }^{p\left( \cdot \right) }}%
}\leq c,\quad v=1-\gamma ,...,1,
\end{equation*}%
by using $\mathrm{\eqref{mod-est}}$ and $\mathrm{\eqref{modular}}$. The
proof is complete.
\end{proof}

\begin{definition}
\label{sequence-space}Let $p,q\in \mathcal{P}_{0}$, $\tau :\mathbb{R}%
^{n}\rightarrow \mathbb{R}^{+}$\ and let $\alpha $ $:\mathbb{R}%
^{n}\rightarrow \mathbb{R}$. Then for all complex valued sequences $\lambda
=\{\lambda _{v,m}\in \mathbb{C}:v\in \mathbb{N}_{0},m\in \mathbb{Z}^{n}\}$
we define%
\begin{equation*}
\mathfrak{b}_{p(\cdot ),q(\cdot )}^{\alpha (\cdot ),\tau (\cdot )}:=\Big\{%
\lambda :\left\Vert \lambda \right\Vert _{\mathfrak{b}_{p(\cdot ),q(\cdot
)}^{\alpha (\cdot ),\tau (\cdot )}}<\infty \Big\},
\end{equation*}%
where%
\begin{equation*}
\left\Vert \lambda \right\Vert _{\mathfrak{b}_{p(\cdot ),q(\cdot )}^{\alpha
(\cdot ),\tau (\cdot )}}:=\sup_{P\in \mathcal{Q}}\Big\|\Big(\frac{%
\sum\limits_{m\in \mathbb{Z}^{n}}2^{v(\alpha \left( \cdot \right) +\frac{n}{2%
})}\lambda _{v,m}\chi _{v,m}}{|P|^{\tau (\cdot )}}\chi _{P}\Big)_{v\geq
v_{P}^{+}}\Big\|_{\ell ^{q(\cdot )}(L^{p(\cdot )})}.
\end{equation*}
\end{definition}

If we replace dyadic cubes $P$ by arbitrary balls $B_{J}$ of $\mathbb{R}^{n}$
with $J\in \mathbb{Z}$, we then obtain equivalent quasi-norms, where the
supremum is taken over all $J\in \mathbb{Z}$\ and all balls $B_{J}$\ of $%
\mathbb{R}^{n}$.

\begin{lemma}
\label{key-lemmasection3}Let $\alpha ,\mathbb{\tau }\in C_{\mathrm{loc}%
}^{\log }$, $\tau ^{-}\geq 0$, $p,q\in \mathcal{P}_{0}^{\log }$, $%
0<q^{+}<\infty $, $v\in \mathbb{N}_{0},m\in \mathbb{Z}^{n}$, $x\in Q_{v,m}$
and $\lambda \in \mathfrak{b}_{p(\cdot ),q(\cdot )}^{\alpha (\cdot ),\tau
(\cdot )}$. Then there exists $c>0$ independent of $x,v$ and $m$ such that%
\begin{equation*}
|\lambda _{v,m}|\leq c\text{ }2^{-v(\alpha (x)+\frac{n}{2})}|Q_{v,m}|^{\tau
(x)}\left\Vert \lambda \right\Vert _{\mathfrak{b}_{p(\cdot ),q(\cdot
)}^{\alpha (\cdot ),\tau (\cdot )}}\left\Vert \chi _{v,m}\right\Vert
_{p(\cdot )}^{-1}.
\end{equation*}
\end{lemma}

\begin{proof}
Let $\lambda \in \mathfrak{b}_{p(\cdot ),q(\cdot )}^{\alpha (\cdot ),\tau
(\cdot )},v\in \mathbb{N}_{0},m\in \mathbb{Z}^{n}$ and $x\in Q_{v,m}$, with $%
Q_{v,m}\in \mathcal{Q}$. Then 
\begin{equation*}
|\lambda _{v,m}|^{p^{-}}=|Q_{v,m}|^{-1}\int_{Q_{v,m}}|\lambda
_{v,m}|^{p^{-}}\chi _{v,m}(y)dy.
\end{equation*}%
Using the fact that $2^{v(\alpha \left( x\right) -\alpha \left( y\right)
)}\leq c$ and $2^{v(\tau \left( x\right) -\tau \left( y\right) )}\leq c$ for
any $x,y\in Q_{v,m}$, we obtain%
\begin{equation*}
\frac{2^{v(\alpha \left( x\right) +\frac{n}{2})p^{-}}}{|Q_{v,m}|^{p^{-}\tau
(x)}}|\lambda _{v,m}|^{p^{-}}\lesssim |Q_{v,m}|^{-1}\int_{Q_{v,m}}\frac{%
2^{v(\alpha \left( y\right) +\frac{n}{2})p^{-}}}{|Q_{v,m}|^{p^{-}\tau (y)}}%
|\lambda _{v,m}|^{p^{-}}\chi _{v,m}(y)dy.
\end{equation*}%
Applying H\"{o}lder's inequality to estimate this expression by 
\begin{eqnarray*}
&&c|Q_{v,m}|^{-1}\Big\|\frac{2^{v(\alpha \left( \cdot \right) +\frac{n}{2}%
)p^{-}}}{|Q_{v,m}|^{p^{-}\tau (\cdot )}}|\lambda _{v,m}|^{p^{-}}\chi _{v,m}%
\Big\|_{\frac{p}{p^{-}}}\left\Vert \chi _{v,m}\right\Vert _{(\frac{p}{p^{-}}%
)^{\prime }} \\
&\lesssim &\left\Vert \lambda \right\Vert _{\mathfrak{b}_{p(\cdot ),q(\cdot
)}^{\alpha (\cdot ),\tau (\cdot )}}^{p^{-}}\left\Vert \chi _{v,m}\right\Vert
_{\frac{p}{p^{-}}}^{-1},
\end{eqnarray*}%
where we have used \eqref{DHHR}. Therefore for any $x\in Q_{v,m}$%
\begin{equation*}
|\lambda _{v,m}|\lesssim \text{ }2^{-v(\alpha (x)+\frac{n}{2}%
)}|Q_{v,m}|^{\tau (x)}\left\Vert \lambda \right\Vert _{\mathfrak{b}_{p(\cdot
),q(\cdot )}^{\alpha (\cdot ),\tau (\cdot )}}\left\Vert \chi
_{v,m}\right\Vert _{p(\cdot )}^{-1},
\end{equation*}%
which completes the proof.
\end{proof}

As in \cite{D6}, and using Lemma \ref{key-lemmasection3} we obtain the
following statement.

\begin{lemma}
Let $\alpha ,\mathbb{\tau }\in C_{\mathrm{loc}}^{\log }$, $\mathbb{\tau }%
^{-}\geq 0$, $p,q\in \mathcal{P}_{0}^{\log }$ and $\Psi $, $\psi \in 
\mathcal{S}(\mathbb{R}^{n})$ satisfy, respectively, \eqref{Ass1} and %
\eqref{Ass2}. Then for all $\lambda \in \mathfrak{b}_{p(\cdot ),q(\cdot
)}^{\alpha (\cdot ),\tau (\cdot )}$%
\begin{equation*}
T_{\psi }\lambda :=\sum_{m\in \mathbb{Z}^{n}}\lambda _{0,m}\Psi
_{m}+\sum_{v=1}^{\infty }\sum_{m\in \mathbb{Z}^{n}}\lambda _{v,m}\psi _{v,m},
\end{equation*}%
converges in $\mathcal{S}^{\prime }(\mathbb{R}^{n})$; moreover, $T_{\psi }:%
\mathfrak{b}_{p(\cdot ),q(\cdot )}^{\alpha (\cdot ),\tau (\cdot
)}\rightarrow \mathcal{S}^{\prime }(\mathbb{R}^{n})$ is continuous.
\end{lemma}

For a sequence $\lambda =\{\lambda _{v,m}\in \mathbb{C}:v\in \mathbb{N}%
_{0},m\in \mathbb{Z}^{n}\},0<r\leq \infty $ and a fixed $d>0$, set%
\begin{equation*}
\lambda _{v,m,r,d}^{\ast }:=\Big(\sum_{h\in \mathbb{Z}^{n}}\frac{|\lambda
_{v,h}|^{r}}{(1+2^{v}|2^{-v}h-2^{-v}m|)^{d}}\Big)^{\frac{1}{r}}
\end{equation*}%
and $\lambda _{r,d}^{\ast }:=\{\lambda _{v,m,r,d}^{\ast }\in \mathbb{C}:v\in 
\mathbb{N}_{0},m\in \mathbb{Z}^{n}\}$.

The proof of the following lemma is postponed to the Appendix.

\begin{lemma}
\label{lamda-equi}Let $\alpha ,\mathbb{\tau }\in C_{\mathrm{loc}}^{\log }$, $%
\tau ^{-}>0$, $p,q\in \mathcal{P}_{0}^{\log }$, $0<q^{+}<\infty $ and $0<r<%
\frac{(\tau p)^{-}}{\tau ^{+}}$. Then for $d$ large enough%
\begin{equation*}
\left\Vert \lambda _{r,d}^{\ast }\right\Vert _{\mathfrak{b}_{p(\cdot
),q(\cdot )}^{\alpha (\cdot ),\tau (\cdot )}}\approx \left\Vert \lambda
\right\Vert _{\mathfrak{b}_{p(\cdot ),q(\cdot )}^{\alpha (\cdot ),\tau
(\cdot )}}.
\end{equation*}
\end{lemma}

By this result, Lemma \ref{new-equinorm3} and by the same arguments given in 
\cite[Theorem 3.14]{D6} we obtain the following statement.

\begin{theorem}
\label{phi-tran}Let $\alpha ,\mathbb{\tau }\in C_{\mathrm{loc}}^{\log }$, $%
\tau ^{-}>0$, $p,q\in \mathcal{P}_{0}^{\log }$ and $0<q^{+}<\infty $. 
\textit{Suppose that }$\Phi $, $\Psi \in \mathcal{S}(\mathbb{R}^{n})$
satisfying \eqref{Ass1} and $\varphi ,\psi \in \mathcal{S}(\mathbb{R}^{n})$
satisfy \eqref{Ass2} such that \eqref{Ass4} holds. The operators $S_{\varphi
}:\mathfrak{B}_{p\left( \cdot \right) ,q\left( \cdot \right) }^{\alpha
\left( \cdot \right) ,\tau (\cdot )}\rightarrow \mathfrak{b}_{p(\cdot
),q(\cdot )}^{\alpha (\cdot ),\tau (\cdot )}$ and $T_{\psi }:\mathfrak{b}%
_{p\left( \cdot \right) ,q\left( \cdot \right) }^{\alpha \left( \cdot
\right) ,\tau (\cdot )}\rightarrow \mathfrak{B}_{p\left( \cdot \right)
,q\left( \cdot \right) }^{\alpha \left( \cdot \right) ,\tau (\cdot )}$ are
bounded. Furthermore, $T_{\psi }\circ S_{\varphi }$ is the identity on $%
\mathfrak{B}_{p\left( \cdot \right) ,q\left( \cdot \right) }^{\alpha \left(
\cdot \right) ,\tau (\cdot )}$.
\end{theorem}

From Theorem \ref{phi-tran}, we obtain the next important property of spaces 
$\mathfrak{B}_{p\left( \cdot \right) ,q\left( \cdot \right) }^{\alpha \left(
\cdot \right) ,\tau (\cdot )}$.

\begin{corollary}
Let $\alpha ,\mathbb{\tau }\in C_{\mathrm{loc}}^{\log }$, $\tau ^{-}>0$, $%
p,q\in \mathcal{P}_{0}^{\log }$ and $0<q^{+}<\infty $, The definition of the
spaces $\mathfrak{B}_{p\left( \cdot \right) ,q\left( \cdot \right) }^{\alpha
\left( \cdot \right) ,\tau (\cdot )}$ is independent of the choices of $\Phi 
$ and $\varphi $.
\end{corollary}

\section{Embeddings}

For the spaces $\mathfrak{B}_{p(\cdot ),q(\cdot )}^{\alpha (\cdot ),\tau
(\cdot )}$ introduced above we want to show some embedding theorems. We say
a quasi-Banach space $A_{1}$ is continuously embedded in another
quasi-Banach space $A_{2}$, $A_{1}\hookrightarrow A_{2}$, if $A_{1}\subset
A_{2}$ and there is a $c>0$ such that $\left\Vert f\right\Vert _{A_{2}}\leq
c\left\Vert f\right\Vert _{A_{1}}$ for all $f\in A_{1}$. We begin with the
following elementary embeddings.

\begin{theorem}
\label{Elem-emb}Let $\alpha ,\mathbb{\tau }\in C_{\mathrm{loc}}^{\log }$, $%
\tau ^{-}>0$ and $p,q,q_{0},q_{1}\in \mathcal{P}_{0}^{\log }$ with $%
p^{+},q^{+},q_{0}^{+},q_{1}^{+}<\infty $.\newline
(i) If $q_{0}\leq q_{1}$, then 
\begin{equation*}
\mathfrak{B}_{p(\cdot ),q_{0}(\cdot )}^{\alpha (\cdot ),\tau (\cdot
)}\hookrightarrow \mathfrak{B}_{p(\cdot ),q_{1}(\cdot )}^{\alpha (\cdot
),\tau (\cdot )}.
\end{equation*}%
(ii) If $(\alpha _{0}-\alpha _{1})^{-}>0$, then 
\begin{equation*}
\mathfrak{B}_{{p(\cdot )},q_{0}{(\cdot )}}^{\alpha _{0}(\cdot ){,\tau (\cdot
)}}\hookrightarrow \mathfrak{B}_{{p(\cdot )},q_{1}{(\cdot )}}^{\alpha
_{1}(\cdot ){,\tau (\cdot )}}.
\end{equation*}
\end{theorem}

The proof can be obtained by using the same method as in \cite[Theorem 6.1]%
{AH}. We next consider embeddings of Sobolev-type. It is well-known that%
\begin{equation*}
B_{{p}_{0},q}^{{\alpha }_{0}{,\tau }}\hookrightarrow B_{{p}_{1},q}^{{\alpha }%
_{1}{,\tau }},
\end{equation*}%
if ${\alpha }_{0}-\frac{n}{{p}_{0}}={\alpha }_{1}-\frac{n}{{p}_{1}}$, where $%
0<{p}_{0}<{p}_{1}\leq \infty ,0\leq {\tau }<\infty $ and $0<q\leq \infty $
(see e.g. \cite[Corollary 2.2]{WYY}). In the following theorem we generalize
these embeddings to variable exponent case.

\begin{theorem}
\label{Sobolev-emb}Let $\alpha _{0},\alpha _{1},\mathbb{\tau }\in C_{\mathrm{%
loc}}^{\log }$, $\tau ^{-}>0$ and $p_{0},p_{1},q\in \mathcal{P}_{0}^{\log }$
with $q^{+}<\infty $. If ${\alpha }_{0}(\cdot )>{\alpha }_{1}(\cdot )$\ and $%
{\alpha }_{0}(\cdot ){-}\frac{n}{p_{0}(\cdot )}={\alpha }_{1}(\cdot ){-}%
\frac{n}{p_{1}(\cdot )}$ with $\Big(\frac{p_{0}}{p_{1}}\Big)^{-}<1$, then 
\begin{equation*}
\mathfrak{B}_{{p}_{0}{(\cdot )},q{(\cdot )}}^{{\alpha }_{0}{(\cdot ),\tau
(\cdot )}}\hookrightarrow \mathfrak{B}_{{p}_{1}{(\cdot )},q{(\cdot )}}^{{%
\alpha }_{1}{(\cdot ),\tau (\cdot )}}.
\end{equation*}
\end{theorem}

\begin{proof}
Let $f\in \mathfrak{B}_{{p}_{0}{(\cdot )},q{(\cdot )}}^{{\alpha }_{0}{(\cdot
),\tau (\cdot )}}$ and $P$ be any dyadic cube of $\mathbb{R}^{n}$.

\textit{Case 1.} $l(P)>1$. Let $Q_{v}\subset P$ be a cube, with $\ell \left(
Q_{v}\right) =2^{-v}$ and $x\in Q_{v}\subset P$. By Lemma \ref{r-trick} we
have for any $m>n$, $d>0$ 
\begin{equation*}
|\varphi _{v}\ast f(x)|\leq c(\eta _{v,m}\ast |\varphi _{v}\ast f|^{d}(x))^{%
\frac{1}{d}}.
\end{equation*}%
We have%
\begin{eqnarray*}
&&\eta _{v,m}\ast |\varphi _{v}\ast f|^{d}(x) \\
&=&2^{vn}\int_{\mathbb{R}^{n}}\frac{\left\vert \varphi _{v}\ast
f(z)\right\vert ^{d}}{\left( 1+2^{v}\left\vert x-z\right\vert \right) ^{m}}dz
\\
&=&\int_{3Q_{v}}\cdot \cdot \cdot dz+\sum_{k\in \mathbb{Z}^{n},\left\Vert
k\right\Vert _{\infty }\geq 2}\int_{Q_{v}^{k}}\cdot \cdot \cdot dz,
\end{eqnarray*}%
where $Q_{v}^{k}=Q_{v}+kl(Q_{v})$. Let $z\in Q_{v}^{k}$ with $k\in \mathbb{Z}%
^{n}$ and $|k|>4\sqrt{n}$. Then $\left\vert x-z\right\vert \geq \left\vert
k\right\vert 2^{-v-1}$ and the second integral is bounded by 
\begin{equation*}
c\text{ }\left\vert k\right\vert ^{-m}M_{Q_{v}^{k}}\left( |\varphi _{v}\ast
f|^{d}\right) ,
\end{equation*}%
where the positive constant $c$ independent of $k$ and $v$. Fix 
\begin{equation*}
0<d\tau ^{+}<r<\frac{1}{2}\min (\frac{p^{-}}{d},\frac{q^{-}}{d},2,(p_{0}\tau
)^{-}),
\end{equation*}%
we have 
\begin{eqnarray}
&&\Big\|\Big(\frac{2^{v\alpha _{1}(\cdot )}\varphi _{v}\ast f}{\left\vert
P\right\vert ^{\tau (\cdot )}}\chi _{P}\Big)_{v\in \mathbb{N}_{0}}\Big\|%
_{\ell ^{q(\cdot )}(L^{p_{1}(\cdot )})}^{rd}  \notag \\
&\lesssim &\Big\|\Big(\frac{2^{v\alpha _{1}(\cdot )}}{\left\vert
P\right\vert ^{\tau (\cdot )}}\left( M_{3Q_{v}}\left( |\varphi _{v}\ast
f|^{d}\right) \right) ^{\frac{1}{d}}\chi _{P}\Big)_{v\in \mathbb{N}_{0}}%
\Big\|_{\ell ^{q(\cdot )}(L^{p_{1}(\cdot )})}^{rd}  \label{term1} \\
&&+\sum_{k\in \mathbb{Z}^{n},\left\Vert k\right\Vert _{\infty }\geq
2}\left\vert k\right\vert ^{\sigma rd}  \notag \\
&&\times \Big\|\Big(\frac{2^{v\alpha _{1}(\cdot )}\left\vert k\right\vert
^{b(\cdot )}}{\left\vert P\right\vert ^{\tau (\cdot )}}(M_{Q_{v}^{k}}\left(
|\varphi _{v}\ast f|^{d}\right) )^{\frac{1}{d}}\chi _{P}\Big)_{v\in \mathbb{N%
}_{0}}\Big\|_{\ell ^{q(\cdot )}(L^{p_{1}(\cdot )})}^{rd},
\label{second-norm1}
\end{eqnarray}%
where 
\begin{equation*}
b\left( \cdot \right) =-\frac{sn\tau (\cdot )}{r}-2c_{\log }(\frac{1}{q\tau }%
)\tau (\cdot )-c_{\log }\big(\alpha _{1}-\frac{n}{p_{1}}\big)-\frac{2n}{d}%
-nc_{\log }(\frac{n}{p_{0}})
\end{equation*}%
and 
\begin{equation*}
\sigma =\frac{sn\tau ^{+}}{r}+2c_{\log }(\frac{1}{q\tau })\tau ^{+}+c_{\log }%
\big(\alpha _{1}-\frac{n}{p_{1}}\big)+\frac{2n}{d}+nc_{\log }(\frac{n}{p_{0}}%
)-m,
\end{equation*}%
where $s$ will be choosen later.

\textit{Estimate of} \eqref{term1}. We will prove that \eqref{term1}, with
power $\frac{1}{rd}$, is bounded by 
\begin{equation}
c\Big\|\Big(\frac{2^{v\alpha _{1}(\cdot )}\varphi _{v}\ast f}{\left\vert
P\right\vert ^{\tau (\cdot )}}\chi _{3P}\Big)_{v\in \mathbb{N}_{0}}\Big\|%
_{\ell ^{q(\cdot )}(L^{p_{1}(\cdot )})}\lesssim \left\Vert f\right\Vert _{%
\mathfrak{B}_{{p}_{0}{(\cdot )},q{(\cdot )}}^{{\alpha }_{0}{(\cdot ),\tau
(\cdot )}}}.  \label{est-Bcase}
\end{equation}%
By the scaling argument, we see that it suffices to consider the case when
the left-hand side is less than or equal $1$. Therefore we will prove that%
\begin{equation*}
\sum\limits_{v=0}^{\infty }\Big\|\Big|\frac{c\text{ }2^{v\alpha _{1}(\cdot )}%
}{\left\vert P\right\vert ^{\tau (\cdot )}}\left( M_{3Q_{v}}|\varphi
_{v}\ast f|^{d}\right) ^{\frac{1}{d}}\chi _{P}\Big|^{q(\cdot )}\Big\|_{\frac{%
p_{1}(\cdot )}{q(\cdot )}}\lesssim 1
\end{equation*}%
for some positive constant $c>0$. This clearly follows from the inequality%
\begin{eqnarray*}
\Big\|\Big|\frac{c\text{ }2^{v\alpha _{1}(\cdot )}}{\left\vert P\right\vert
^{\tau (\cdot )}}(M_{3Q_{v}}\left( |\varphi _{v}\ast f|^{d}\right) )^{\frac{1%
}{d}}\chi _{P}\Big|^{q(\cdot )}\Big\|_{\frac{p_{1}(\cdot )}{q(\cdot )}}
&\leq &\Big\|\Big|\frac{2^{v\alpha _{0}(\cdot )}\varphi _{v}\ast f}{%
\left\vert P\right\vert ^{\tau (\cdot )}}\Big|^{q(\cdot )}\chi _{3P}\Big\|_{%
\frac{p_{0}(\cdot )}{q(\cdot )}}+2^{-v} \\
&=&\delta .
\end{eqnarray*}%
This claim can be reformulated as showing that%
\begin{equation*}
\Big\|\Big|\frac{c\text{ }\delta ^{-\frac{1}{q\left( \cdot \right) }%
}2^{v\alpha _{1}(\cdot )}}{\left\vert P\right\vert ^{\tau (\cdot )}}%
(M_{3Q_{v}}\left( |\varphi _{v}\ast f|^{d}\right) ^{\frac{1}{d}}\Big|%
^{q(\cdot )}\chi _{P}\Big\|_{\frac{p_{1}(\cdot )}{q\left( \cdot \right) }%
}\leq 1,
\end{equation*}%
which is equivalent to%
\begin{equation}
\int_{P}\frac{\delta ^{-\frac{p_{1}(x)}{q\left( x\right) }}2^{v\alpha
_{1}(x)p_{1}(x)}}{\left\vert P\right\vert ^{\tau (x)p_{1}(x)}}\big(%
M_{3Q_{v}}(|\varphi _{v}\ast f|^{d})\big)^{\frac{p_{1}(x)}{d}}dx\lesssim 1.
\label{est-Int}
\end{equation}%
Since $\alpha _{1}$ and $p_{1}$ are log-H\"{o}lder continuous, we can move $%
2^{v(\alpha _{1}(x)-\frac{n}{p_{1}(x)})}$ inside the integral by Lemma \ref%
{DHR-lemma}: 
\begin{equation}
\delta ^{-\frac{1}{q\left( x\right) }}\frac{2^{v(\alpha _{1}(x)-\frac{n}{%
p_{1}(x)})}}{\left\vert P\right\vert ^{\tau (x)}}\left( M_{3Q_{v}}\left(
|\varphi _{v}\ast f|^{d}\right) \right) ^{\frac{1}{d}}\lesssim \frac{\delta
^{-\frac{1}{q\left( x\right) }}}{\left\vert P\right\vert ^{\tau (x)}}\Big(%
M_{3Q_{v}}\big(2^{v(\alpha _{1}(\cdot )-\frac{n}{p_{1}(\cdot )})d}|\varphi
_{v}\ast f|^{d}\big)\Big)^{\frac{1}{d}}  \label{est3}
\end{equation}%
for any $x\in Q_{v}\subset P$. Observe that 
\begin{equation*}
0<d<\min \big(\frac{p^{-}}{2r},\frac{q^{-}}{2r},\frac{r}{\tau ^{+}}\big).
\end{equation*}%
The right-hand side of \eqref{est3} can be rewritten us%
\begin{equation}
\Big(\frac{1}{\left\vert P\right\vert ^{r}}\Big(\delta ^{-\frac{d}{q\left(
x\right) }}M_{3Q_{v}}\big(2^{v(\alpha _{1}(\cdot )-\frac{n}{p_{1}(\cdot )}%
)d}|\varphi _{v}\ast f|^{d}\big)\Big)^{\frac{r}{d\tau (x)}}\Big)^{\frac{\tau
(x)}{r}}.  \label{key-est-emb}
\end{equation}%
By Lemma \ref{DHHR-estimate}, Remark \ref{new-est} and since $\frac{1}{q}$
and $\tau $ are log-H\"{o}lder continuous,%
\begin{equation*}
\delta ^{-\frac{r}{q\left( x\right) \tau \left( x\right) }}\Big(\frac{\beta 
}{\left\vert 3Q_{v}\right\vert }\int_{3Q_{v}}2^{v(\alpha _{1}(y)-\frac{n}{%
p_{1}(y)})d}|\varphi _{v}\ast f(y)|^{d}dy\Big)^{\frac{r}{d\tau (x)}}
\end{equation*}%
can be estimated by 
\begin{eqnarray*}
&&\frac{c}{\left\vert 3Q_{v}\right\vert }\int_{3Q_{v}}\delta ^{-\frac{r}{%
q\left( y\right) \tau \left( y\right) }}2^{\frac{vr(\alpha _{1}(y)-\frac{n}{%
p_{1}(y)})}{\tau \left( y\right) }}|\varphi _{v}\ast f(y)|^{\frac{r}{\tau
\left( y\right) }}dy+\left\vert Q_{v}\right\vert ^{s}g(x) \\
&\lesssim &\int_{3Q_{v}}2^{vn}\delta ^{-\frac{r}{q\left( y\right) \tau
\left( y\right) }}2^{\frac{v(\alpha _{1}(y)-\frac{n}{p_{1}(y)})r}{\tau
\left( y\right) }}|\varphi _{v}\ast f(y)|^{\frac{r}{\tau \left( y\right) }%
}dy+h(x)
\end{eqnarray*}%
for any $s>0$ large enough where%
\begin{equation*}
g(x)=\left( e+\left\vert x\right\vert \right) ^{-s}+M_{3Q_{v}}\left( \left(
e+\left\vert \cdot \right\vert \right) ^{-s}\right) ,\quad x\in \mathbb{R}%
^{n},s>0
\end{equation*}%
and%
\begin{equation*}
h(x)=\left( e+\left\vert x\right\vert \right) ^{-s}+\mathcal{M}\left( \left(
e+\left\vert \cdot \right\vert \right) ^{-s}\right) \left( x\right) ,\quad
x\in \mathbb{R}^{n},s>0.
\end{equation*}%
These two functions will be used throughout the paper. Therefore %
\eqref{key-est-emb}, with power $\frac{1}{\tau (x)}$, is bounded by%
\begin{equation*}
\Big\|\frac{\delta ^{-\frac{r}{q\left( \cdot \right) \tau \left( \cdot
\right) }}2^{v\frac{\alpha _{0}(\cdot )r}{\tau \left( \cdot \right) }%
}|\varphi _{v}\ast f|^{\frac{r}{\tau \left( \cdot \right) }}}{\left\vert
P\right\vert ^{r}}\chi _{3P}\Big\|_{\frac{p_{0}(\cdot )\tau \left( \cdot
\right) }{r}}^{\frac{1}{r}}\big\|2^{v\frac{n}{t(\cdot )}}\chi _{3Q_{v}}\big\|%
_{t(\cdot )}^{\frac{1}{r}}+c,
\end{equation*}%
by H\"{o}lder's inequality, with $1=\frac{r}{p_{0}(\cdot )\tau \left( \cdot
\right) }+\frac{1}{t(\cdot )}$. The second norm is bounded and the first
norm is bounded if and only if%
\begin{equation*}
\Big\|\frac{\delta ^{-\frac{1}{q\left( \cdot \right) }}2^{v\alpha _{0}(\cdot
)}\left\vert \varphi _{v}\ast f\right\vert \chi _{3P}}{|P|^{\tau (\cdot )}}%
\Big\|_{p_{0}(\cdot )}\lesssim 1,
\end{equation*}%
which follows immediately from the definition of $\delta $.\ Now, we find
that the left-hand side of \eqref{est-Int} can be rewritten as 
\begin{eqnarray*}
&&\int_{P}\Big(\frac{\delta ^{-\frac{1}{q\left( x\right) }}2^{v(\alpha
_{1}(x)-\frac{n}{p_{1}(x)})}}{\left\vert P\right\vert ^{\tau (x)}}%
(M_{3Q_{v}}(|\varphi _{v}\ast f|^{d}))^{\frac{1}{d}}\Big)^{p_{1}(x)-p_{0}(x)}
\\
&&\times \Big(\frac{\delta ^{-\frac{1}{q\left( x\right) }}2^{v(\alpha
_{1}(x)+\frac{n}{p_{0}(x)}-\frac{n}{p_{1}(x)})}}{\left\vert P\right\vert
^{\tau (x)}}(M_{3Q_{v}}(|\varphi _{v}\ast f|^{d}))^{\frac{1}{d}}\Big)%
^{p_{0}(x)}dx \\
&\lesssim &\int_{P}\frac{1}{\left\vert P\right\vert ^{\tau (x)p_{0}(x)}}\Big(%
\delta ^{-\frac{d}{q\left( x\right) }}M_{3Q_{v}}\left( 2^{v\alpha _{0}(\cdot
)d}|\varphi _{v}\ast f|^{d}\right) \Big)^{\frac{p_{0}(x)}{d}}dx.
\end{eqnarray*}%
The last expression is bounded if and only if%
\begin{equation*}
\Big\|\frac{1}{\left\vert P\right\vert ^{r}}\Big(\delta ^{-\frac{d}{q\left(
\cdot \right) }}M_{3Q_{v}}\left( 2^{v\alpha _{0}(\cdot )d}|\varphi _{v}\ast
f|^{d}\right) \chi _{P}\Big)^{\frac{r}{d\tau \left( \cdot \right) }}\Big\|_{%
\frac{p_{0}(\cdot )\tau \left( \cdot \right) }{r}}\lesssim 1.
\end{equation*}%
This norm is bounded by%
\begin{equation*}
\Big\|\mathcal{M}\Big(\frac{\delta ^{-\frac{r}{q\left( \cdot \right) \tau
\left( \cdot \right) }}2^{\frac{v\alpha _{0}(\cdot )r}{\tau \left( \cdot
\right) }}|\varphi _{v}\ast f|^{\frac{r}{\tau \left( \cdot \right) }}}{%
\left\vert P\right\vert ^{r}}\chi _{3P}\Big)\Big\|_{\frac{p_{0}(\cdot )\tau
\left( \cdot \right) }{r}}+c,
\end{equation*}%
where we have used again Lemma \ref{DHHR-estimate} and Remark \ref{new-est}.
Since the maximal function is bounded in $L^{p(\cdot )}$ when $p\in \mathcal{%
P}^{\log }$ and $p^{-}>1$, this expression is bounded by%
\begin{equation*}
\Big\|\frac{\delta ^{-\frac{1}{q\left( \cdot \right) \tau \left( \cdot
\right) }}2^{\frac{v\alpha _{0}(\cdot )}{\tau \left( \cdot \right) }%
}|\varphi _{v}\ast f|^{\frac{1}{\tau \left( \cdot \right) }}\chi _{3P}}{%
\left\vert P\right\vert }\Big\|_{p_{0}(\cdot )\tau \left( \cdot \right)
}^{r}+c.
\end{equation*}%
The last quasi-norm is bounded if and only if%
\begin{equation*}
\Big\|\frac{\delta ^{-\frac{1}{q\left( \cdot \right) }}2^{v\alpha _{0}(\cdot
)}|\varphi _{v}\ast f|\chi _{3P}}{\left\vert P\right\vert ^{\tau \left(
\cdot \right) }}\Big\|_{p_{0}(\cdot )}\lesssim 1.
\end{equation*}%
due to the choice of $\delta $.

\textit{Estimate of} \eqref{second-norm1}. The summation in %
\eqref{second-norm1} can be estimated by%
\begin{equation*}
\sum_{k\in \mathbb{Z},|k|\leq 4\sqrt{n}}\cdot \cdot \cdot +\sum_{k\in 
\mathbb{Z}^{n},|k|>4\sqrt{n}}\cdot \cdot \cdot .
\end{equation*}%
The estimation of the first sum follows in the same manner as before. Let us
prove that 
\begin{equation*}
\Big\|\Big(\frac{2^{v\alpha _{1}(\cdot )}\left\vert k\right\vert ^{b\left(
\cdot \right) }(M_{Q_{v}^{k}}\left( |\varphi _{v}\ast f|^{d}\right) )^{\frac{%
1}{d}}}{|\tilde{Q}^{k}|^{\tau (\cdot )}}\chi _{P}\Big)_{v\in \mathbb{N}_{0}}%
\Big\|_{\ell ^{q(\cdot )}(L^{p_{1}(\cdot )})}\lesssim \big\|f\big\|_{%
\mathfrak{B}_{{p}_{0}{(\cdot )},q{(\cdot )}}^{{\alpha }_{0}{(\cdot ),\tau
(\cdot )}}}
\end{equation*}%
for any $k\in \mathbb{Z}^{n}$ with $|k|>4\sqrt{n}$, where $\tilde{Q}%
^{k}=Q\left( c_{P},2\left\vert k\right\vert l(P)\right) $. By the scaling
argument, we see that it suffices to consider the case when the left-hand
side is less than or equal $1$. Therefore we will prove that%
\begin{equation*}
\sum\limits_{v=0}^{\infty }\Big\|\Big|\frac{2^{v\alpha _{1}(\cdot
)}\left\vert k\right\vert ^{b\left( \cdot \right) -n\tau (\cdot )}}{%
\left\vert P\right\vert ^{\tau (\cdot )}}(M_{Q_{v}^{k}}\left( |\varphi
_{v}\ast f|^{d}\right) )^{\frac{1}{d}}\chi _{P}\Big|^{q(\cdot )}\Big\|_{%
\frac{p_{1}(\cdot )}{q(\cdot )}}\lesssim 1.
\end{equation*}%
This clearly follows from the inequality%
\begin{eqnarray*}
&&\Big\|\Big|\frac{c\text{ }2^{v\alpha _{1}(\cdot )}\left\vert k\right\vert
^{b\left( \cdot \right) -n\tau (\cdot )}}{\left\vert P\right\vert ^{\tau
(\cdot )}}(M_{Q_{v}^{k}}\left( |\varphi _{v}\ast f|^{d}\right) )^{\frac{1}{d}%
}\chi _{P}\Big|^{q(\cdot )}\Big\|_{\frac{p_{1}(\cdot )}{q(\cdot )}} \\
&\leq &\Big\|\Big|\frac{2^{v\alpha _{0}(\cdot )}\varphi _{v}\ast f}{|\tilde{Q%
}^{k}|^{\tau (\cdot )}}\Big|^{q(\cdot )}\chi _{\tilde{Q}^{k}}\Big\|_{\frac{%
p_{0}(\cdot )}{q(\cdot )}}+2^{-v} \\
&=&\delta 
\end{eqnarray*}%
for some positive constant $c$. This claim can be reformulated as showing
that%
\begin{equation*}
\int_{P}\frac{\delta ^{-\frac{p_{1}(x)}{q\left( x\right) }}2^{v\alpha
_{1}(x)p_{1}(x)}\left\vert k\right\vert ^{(b\left( x\right) -n\tau
(x))p_{1}(x)}}{\left\vert P\right\vert ^{\tau (x)p_{1}(x)}}%
(M_{Q_{v}^{k}}\left( |\varphi _{v}\ast f|^{d}\right) )^{\frac{p_{1}(x)}{d}%
}dx\lesssim 1.
\end{equation*}%
Since, again, $\alpha _{1}$ and $p_{1}$ are log-H\"{o}lder continuous, we
can move $2^{v(\alpha _{1}(x)-\frac{n}{p_{1}(x)})}$ inside the integral by
Lemma \ref{DHR-lemma}: 
\begin{eqnarray*}
&&\text{ }\frac{\left\vert k\right\vert ^{-c_{\log }(\alpha _{1}-\frac{n}{%
p_{1}})-\frac{n}{d}}2^{v(\alpha _{1}(x)-\frac{n}{p_{1}(x)})}}{\left\vert
P\right\vert ^{\tau (x)}}(M_{Q_{v}^{k}}\left( |\varphi _{v}\ast
f|^{d}\right) )^{\frac{1}{d}} \\
&\lesssim &\frac{1}{\left\vert P\right\vert ^{\tau (x)}}\Big(M_{Q_{v}^{k}}%
\big(\left\vert k\right\vert ^{-n}2^{v(\alpha _{1}(\cdot )-\frac{n}{%
p_{1}(\cdot )})d}|\varphi _{v}\ast f|^{d}\big)\Big)^{\frac{1}{d}},
\end{eqnarray*}%
where the implicit constant is independnet of $x,v$ and $k$. We have%
\begin{eqnarray}
&&\frac{\left\vert k\right\vert ^{b\left( \cdot \right) -n\tau (\cdot
)}\delta ^{-\frac{1}{q\left( \cdot \right) }}}{\left\vert P\right\vert
^{\tau (\cdot )}}\Big(M_{Q_{v}^{k}}\big(\frac{|\varphi _{v}\ast f|^{d}}{%
\left\vert k\right\vert ^{n}2^{-v(\alpha _{1}(\cdot )-\frac{n}{p_{1}(\cdot )}%
)d}}\big)\Big)^{\frac{1}{d}}  \notag \\
&=&\Big(\frac{\left\vert k\right\vert ^{(b\left( \cdot \right) -n\tau (\cdot
))\frac{r}{\tau (\cdot )}}\delta ^{-\frac{r}{q\left( \cdot \right) \tau
(\cdot )}}}{\left\vert P\right\vert ^{r}}\Big(M_{Q_{v}^{k}}\big(\frac{%
|\varphi _{v}\ast f|^{d}}{\left\vert k\right\vert ^{n}2^{-v(\alpha
_{1}(\cdot )-\frac{n}{p_{1}(\cdot )})d}}\big)\Big)^{\frac{r}{d\tau (\cdot )}}%
\Big)^{\frac{\tau (\cdot )}{r}}.  \label{key-est-emb1}
\end{eqnarray}%
As before, let us prove that this expression, with power $\frac{1}{\tau (x)}$
is bounded. Observe that $Q_{v}^{k}\subset Q(x,\left\vert k\right\vert
2^{-v+1})=\tilde{Q}_{v}^{k}$. We have%
\begin{equation*}
\delta ^{-\frac{1}{q\left( x\right) \tau (x)}}=(2^{v}\delta )^{-\frac{1}{%
q\left( x\right) \tau (x)}+\frac{1}{q\left( y\right) \tau (y)}}(2^{v}\delta
)^{-\frac{1}{q\left( y\right) \tau (y)}}2^{v\frac{1}{q\left( x\right) \tau
(x)}},\quad x\in Q_{v}\subset P,y\in \tilde{Q}_{v}^{k}.
\end{equation*}%
From Lemma \ref{DHR-lemma} it follows that%
\begin{equation*}
2^{v\frac{1}{q\left( x\right) \tau (x)}}\lesssim |k|^{c_{\log }(\frac{1}{%
q\tau })}2^{v\frac{1}{q\left( y\right) \tau (y)}}
\end{equation*}

and%
\begin{equation*}
(2^{v}\delta )^{-\frac{1}{q\left( x\right) \tau (x)}+\frac{1}{q\left(
y\right) \tau (y)}}\lesssim |k|^{c_{\log }(\frac{1}{q\tau })}
\end{equation*}%
for any $x\in Q_{v},y\in \tilde{Q}_{v}^{k}$, where the implicit constant is
independent of $x,y,k$ and $v$. Again by Lemma \ref{DHHR-estimate}\ combined
with Remark \ref{new-est} and since $\frac{1}{q}$ and $\tau $ are log-H\"{o}%
lder continuous,%
\begin{eqnarray*}
&&\left\vert k\right\vert ^{\big(-\frac{sn\tau (x)}{r}-2c_{\log }(\frac{1}{%
q\tau })\tau (x)\big)\frac{r}{\tau (x)}}\delta ^{-\frac{r}{q\left( x\right)
\tau \left( x\right) }} \\
&&\times \Big(\frac{\beta }{|\tilde{Q}_{v}^{k}|}\int_{\tilde{Q}%
_{v}^{k}}2^{v(\alpha _{1}(y)-\frac{n}{p_{1}(y)})d}|\varphi _{v}\ast f\left(
y\right) |^{d}dy\Big)^{\frac{r}{d\tau \left( x\right) }} \\
&\lesssim &\frac{1}{|\tilde{Q}_{v}^{k}|}\int_{\tilde{Q}_{v}^{k}}\delta ^{-%
\frac{r}{q\left( y\right) \tau \left( y\right) }}2^{\frac{vr(\alpha _{1}(y)-%
\frac{n}{p_{1}(y)})}{\tau \left( y\right) }}|\varphi _{v}\ast f\left(
y\right) |^{\frac{r}{\tau \left( y\right) }}dy \\
&&+\left( e+\left\vert x\right\vert \right) ^{-s}+\frac{1}{|\tilde{Q}%
_{v}^{k}|}\int_{\tilde{Q}_{v}^{k}}\left( e+\left\vert y\right\vert \right)
^{-s}dy \\
&\lesssim &\frac{1}{|\tilde{Q}_{v}^{k}|}\int_{\tilde{Q}_{v}^{k}}\delta ^{-%
\frac{r}{q\left( y\right) \tau \left( y\right) }}2^{\frac{v(\alpha _{1}(y)-%
\frac{n}{p_{1}(y)})r}{\tau \left( y\right) }}|\varphi _{v}\ast f\left(
y\right) |^{\frac{r}{\tau \left( y\right) }}dy+h(x)
\end{eqnarray*}%
for any $s>0$ large enough. Therefore the left-hand side of %
\eqref{key-est-emb1}, with power $\frac{1}{\tau (x)}$, is bounded by%
\begin{equation*}
\Big\|\frac{\left\vert k\right\vert ^{-nr}\delta ^{-\frac{r}{q\left( \cdot
\right) \tau \left( \cdot \right) }}2^{\frac{v\alpha _{0}(\cdot )r}{\tau
\left( \cdot \right) }}|\varphi _{v}\ast f|^{\frac{r}{\tau \left( \cdot
\right) }}}{\left\vert P\right\vert ^{r}}\chi _{\tilde{Q}^{k}}\Big\|_{\frac{%
p_{0}(\cdot )\tau \left( \cdot \right) }{r}}^{\frac{1}{r}}\big\|%
2^{vn/t(\cdot )}\chi _{\tilde{Q}_{v}^{k}}\big\|_{t(\cdot )}^{\frac{1}{r}}+c,
\end{equation*}%
by H\"{o}lder's inequality, with $1=\frac{r}{p_{0}(\cdot )\tau \left( \cdot
\right) }+\frac{1}{t(\cdot )}$. As before the second norm is bounded and the
first norm is bounded if and only if%
\begin{equation*}
\int_{\tilde{Q}^{k}}\frac{\delta ^{-\frac{p_{0}(y)}{q\left( y\right) }%
}2^{v\alpha _{0}(y)p_{0}(y)}|\varphi _{v}\ast f\left( y\right) |^{p_{0}(y)}}{%
|\tilde{Q}^{k}|^{p_{0}(\cdot )\tau \left( y\right) }}dy\lesssim 1,
\end{equation*}%
which follows immediately from the definition of $\delta $.\ The desired
estimate, follows using similar arguments as above and by taking $m$ large
enough.

\textit{Case 2.} $l(P)\leq 1$. {Since }$\tau ${\ is log-H\"{o}lder
continuous, we have }%
\begin{equation*}
\left\vert P\right\vert ^{-\tau (x)}\leq c\left\vert P\right\vert ^{-\tau
(y)}(1+2^{v_{P}}\left\vert x-y\right\vert )^{c_{\log }\left( \tau \right)
}\leq c\left\vert P\right\vert ^{-\tau (y)}(1+2^{v}\left\vert x-y\right\vert
)^{c_{\log }\left( \tau \right) }
\end{equation*}%
for any $x,y\in \mathbb{R}^{n}$ and any $v\geq v_{P}$. Therefore,%
\begin{equation*}
\frac{1}{\left\vert P\right\vert ^{\tau (\cdot )d}}\eta _{v,m}\ast \left(
|\varphi _{v}\ast f|^{d}\chi _{3Q_{v}}\right) \lesssim \eta _{v,m-c_{\log
}\left( \tau \right) }\ast \Big(\frac{|\varphi _{v}\ast f|^{d}\chi _{3Q_{v}}%
}{\left\vert P\right\vert ^{\tau (\cdot )d}}\Big)
\end{equation*}%
and%
\begin{equation*}
\frac{1}{\left\vert P\right\vert ^{\tau (\cdot )d}}\eta _{v,m}\ast \left(
|\varphi _{v}\ast f|^{d}\chi _{Q_{v}^{k}}\right) \lesssim \eta _{v,m-c_{\log
}\left( \tau \right) }\ast \Big(\frac{|\varphi _{v}\ast f|^{d}\chi
_{Q_{v}^{k}}}{\left\vert P\right\vert ^{\tau (\cdot )d}}\Big).
\end{equation*}%
The arguments here are quite similar to those used in the case $l(P)>1$,
where we did not need to use Theorem \ref{DHHR-estimate}, which could be
used only to move $\left\vert P\right\vert ^{\tau (\cdot )}$ inside the
convolution and hence the proof is complete.
\end{proof}

\begin{remark}
$\mathrm{We}$ $\mathrm{would}$ $\mathrm{like}$ $\mathrm{to}$ $\mathrm{mention%
}$ $\mathrm{that}$ $\mathrm{similar}$ $\mathrm{arguments}$ $\mathrm{give}$%
\begin{equation*}
\mathfrak{B}_{{p}_{0}{(\cdot )},q{(\cdot )}}^{{\alpha }_{0}{(\cdot ),\tau
(\cdot )}}\hookrightarrow \mathfrak{B}_{{\infty },q{(\cdot )}}^{{\alpha }_{0}%
{(\cdot )-}\frac{n}{{p}_{0}{(\cdot )}}{,\tau (\cdot )}}
\end{equation*}%
$\mathrm{if}$ $\alpha _{0},\mathbb{\tau }\in C_{\mathrm{loc}}^{\log },\tau
^{-}>0$ $\mathrm{and}$ $p_{0},q,\tau \in \mathcal{P}_{0}^{\log }$, with $%
q^{+}<\infty $.
\end{remark}

Let $\alpha ,\mathbb{\tau }\in C_{\mathrm{loc}}^{\log },\tau ^{-}>0,p,q\in 
\mathcal{P}_{0}^{\log }$. From \eqref{emd}, we obtain%
\begin{equation*}
\mathfrak{B}_{p(\cdot ),q(\cdot )}^{\alpha (\cdot ),\tau (\cdot
)}\hookrightarrow B_{p(\cdot ),\infty }^{\alpha (\cdot )+n\tau (\cdot )-%
\frac{n}{p(\cdot )}}\hookrightarrow \mathcal{S}^{\prime }(\mathbb{R}^{n}).
\end{equation*}%
Similar arguments of \cite[Proposition 2.3]{WYY} can be used to prove that%
\begin{equation*}
\mathcal{S}(\mathbb{R}^{n})\hookrightarrow \mathfrak{B}_{p(\cdot ),q(\cdot
)}^{\alpha (\cdot ),\tau (\cdot )}.
\end{equation*}%
Therefore, we obtain the following statement.

\begin{theorem}
Let $\alpha ,\mathbb{\tau }\in C_{\mathrm{loc}}^{\log },\tau ^{-}>0$ and $%
p,q\in \mathcal{P}_{0}^{\log }$ with $q^{+}<\infty $. Then 
\begin{equation*}
\mathcal{S}(\mathbb{R}^{n})\hookrightarrow \mathfrak{B}_{p(\cdot ),q(\cdot
)}^{\alpha (\cdot ),\tau (\cdot )}\hookrightarrow \mathcal{S}^{\prime }(%
\mathbb{R}^{n}).
\end{equation*}
\end{theorem}

Now we establish some further embedding of the spaces $\mathfrak{B}_{p(\cdot
),q(\cdot )}^{\alpha (\cdot ),\tau (\cdot )}$.

\begin{theorem}
\textit{Let }$\alpha ,\tau \in C_{\mathrm{loc}}^{\log },\tau ^{-}>0$ \textit{%
and }$p,q\in \mathcal{P}_{0}^{\log }$ with $q^{+}<\infty $\textit{.} If $%
(p_{2}-p_{1})^{+}\leq 0$, then%
\begin{equation*}
\mathfrak{B}_{{p}_{2}{(\cdot )},q{(\cdot )}}^{{\alpha (\cdot )+n\tau (\cdot
)+\frac{n}{{p}_{2}{(\cdot )}}-\frac{n}{{p}_{1}{(\cdot )}}}}\hookrightarrow 
\mathfrak{B}_{{p}_{1}{(\cdot )},q{(\cdot )}}^{{\alpha (\cdot ),\tau (\cdot )}%
}.
\end{equation*}
\end{theorem}

\begin{proof}
Using the Sobolev embeddings%
\begin{equation*}
B_{{p}_{2}{(\cdot )},q{(\cdot )}}^{{\alpha (\cdot )+n\tau (\cdot )+\frac{n}{{%
p}_{2}{(\cdot )}}-\frac{n}{{p}_{1}{(\cdot )}}}}\hookrightarrow B_{{p}_{1}{%
(\cdot )},q{(\cdot )}}^{{\alpha (\cdot )+n\tau (\cdot )}},
\end{equation*}%
see \cite[Theorem 6.4]{AH} it is sufficient to prove that $B_{{p}_{1}{(\cdot
)},q{(\cdot )}}^{{\alpha (\cdot )+n\tau (\cdot )}}\hookrightarrow \mathfrak{B%
}_{{p}_{1}{(\cdot )},q{(\cdot )}}^{{\alpha (\cdot ),\tau (\cdot )}}$. We
have 
\begin{equation*}
\sup_{P\in \mathcal{Q},|P|>1}\Big\|\Big(\frac{2^{v\alpha \left( \cdot
\right) }\varphi _{v}\ast f}{\left\vert P\right\vert ^{\tau (\cdot )}}\chi
_{P}\Big)_{v\geq v_{P}^{+}}\Big\|_{\ell ^{q(\cdot )}(L^{{p}_{1}{(\cdot )}%
})}\leq \big\|\left( 2^{v\alpha \left( \cdot \right) }\varphi _{v}\ast
f\right) _{v\in \mathbb{N}_{0}}\big\|_{\ell ^{q(\cdot )}(L^{{p}_{1}{(\cdot )}%
})}.
\end{equation*}%
In view of the definition of $B_{{p}_{1}{(\cdot )},q{(\cdot )}}^{{\alpha
(\cdot )}}$ spaces the last expression is bounded by 
\begin{equation*}
\big\|f\big\|_{B_{{p}_{1}{(\cdot )},q(\cdot )}^{\alpha (\cdot )}}\leq \big\|f%
\big\|_{B_{{p}_{1}{(\cdot )},q(\cdot )}^{\alpha (\cdot ){+n\tau (\cdot )}}}.
\end{equation*}%
Now we have the estimates 
\begin{eqnarray*}
&&\sup_{P\in \mathcal{Q},|P|\leq 1}\Big\|\Big(\frac{2^{v\alpha \left( \cdot
\right) }\varphi _{v}\ast f}{\left\vert P\right\vert ^{\tau (\cdot )}}\chi
_{P}\Big)_{v\geq v_{P}^{+}}\Big\|_{\ell ^{q(\cdot )}(L^{{p}_{1}{(\cdot )}})}
\\
&\leq &\sup_{P\in \mathcal{Q},|P|\leq 1}\big\|\left( 2^{v(\alpha \left(
\cdot \right) {+n\tau (\cdot ))+n\tau (\cdot )(v_{P}-v)}}\varphi _{v}\ast
f\right) _{v\geq v_{P}}\big\|_{\ell ^{q(\cdot )}(L^{{p}_{1}{(\cdot )}})} \\
&\leq &\sup_{P\in \mathcal{Q},|P|\leq 1}\big\|\left( 2^{v(\alpha \left(
\cdot \right) {+n\tau (\cdot ))}}\varphi _{v}\ast f\right) _{v\in \mathbb{N}%
_{0}}\big\|_{\ell ^{q(\cdot )}(L^{{p}_{1}{(\cdot )}})} \\
&\leq &\big\|f\big\|_{B_{{p}_{1}{(\cdot )},q(\cdot )}^{\alpha (\cdot ){%
+n\tau (\cdot )}}},
\end{eqnarray*}%
which completes the proof.
\end{proof}

\section{Atomic decomposition}

The idea of atomic decompositions leads back to M. Frazier and B. Jawerth in
their series of papers \cite{FJ86}, \cite{FJ90}. The main goal of this
section is to prove an atomic decomposition result for $\mathfrak{B}%
_{p(\cdot ),q(\cdot )}^{\alpha (\cdot ),\tau (\cdot )}$. We define for $a>0$%
, $\alpha :\mathbb{R}^{n}\rightarrow \mathbb{R}$ and $f\in \mathcal{S}%
^{\prime }(\mathbb{R}^{n})$, the Peetre maximal function%
\begin{equation*}
\varphi _{v}^{\ast ,a}2^{v\alpha (\cdot )}f(x)=\sup_{y\in \mathbb{R}^{n}}%
\frac{2^{v\alpha (y)}\left\vert \varphi _{v}\ast f(y)\right\vert }{\left(
1+2^{v}\left\vert x-y\right\vert \right) ^{a}},\qquad v\in \mathbb{N}_{0}.
\end{equation*}%
where $\varphi _{0}$ is replaced by $\Phi $. We now present a fundamental
characterization of spaces under consideration.

\begin{theorem}
\label{fun-char}Let $\mathbb{\tau }$ $\alpha \in C_{\mathrm{loc}}^{\log
},\tau ^{-}>0$ and $p,q\in \mathcal{P}_{0}^{\log }$. Let $m$ be as in Lemma %
\ref{Alm-Hastolemma1}, $a>\frac{m\tau ^{+}}{(\tau p)^{-}}$ and $\Phi $ and $%
\varphi $ satisfy \eqref{Ass1} and \eqref{Ass2}, respectively. Then%
\begin{equation}
\left\Vert f\right\Vert _{\mathfrak{B}_{p(\cdot ),q(\cdot )}^{\alpha (\cdot
),\tau (\cdot )}}^{\blacktriangledown }:=\sup_{P\in \mathcal{Q}}\Big\|\Big(%
\frac{\varphi _{v}^{\ast ,a}2^{v\alpha (\cdot )}f}{\left\vert P\right\vert
^{\tau (\cdot )}}\chi _{P}\Big)_{v\geq v_{P}^{+}}\Big\|_{\ell ^{q(\cdot
)}(L^{p(\cdot )})}  \label{B-equinorm1}
\end{equation}%
\textit{is an equivalent quasi-norm in }$\mathfrak{B}_{p(\cdot ),q(\cdot
)}^{\alpha (\cdot ),\tau (\cdot )}$.
\end{theorem}

\begin{proof}
We divide the proof in two steps.

\textit{Step }1\textit{. }It is easy to see that for any $f\in \mathcal{S}%
^{\prime }(\mathbb{R}^{n})$ with $\left\Vert f\right\Vert _{\mathfrak{B}%
_{p(\cdot ),q(\cdot )}^{\alpha (\cdot ),\tau (\cdot )}}^{\blacktriangledown
}<\infty $ and any $x\in \mathbb{R}^{n}$ we have%
\begin{equation*}
2^{v\alpha (x)}\left\vert \varphi _{v}\ast f(x)\right\vert \leq \varphi
_{v}^{\ast ,a}2^{v\alpha (\cdot )}f(x)\text{.}
\end{equation*}%
This shows that the right-hand side in \eqref{B-def} is less than or equal %
\eqref{B-equinorm1}.\vskip5pt

\textit{Step }2\textit{.} We will prove in this step that there is a
constant $C>0$ such that for every $f\in \mathfrak{B}_{p(\cdot ),q(\cdot
)}^{\alpha (\cdot ),\tau (\cdot )}$%
\begin{equation}
\big\|f\big\|_{\mathfrak{B}_{p(\cdot ),q(\cdot )}^{\alpha (\cdot ),\tau
(\cdot )}}^{\blacktriangledown }\leq C\big\|f\big\|_{\mathfrak{B}_{p(\cdot
),q(\cdot )}^{\alpha (\cdot ),\tau (\cdot )}}.  \label{estimate-B}
\end{equation}%
We choose $t>0$ such that $a>\frac{m}{t}>\frac{m}{p^{-}}$. By Lemmas \ref%
{r-trick}\ and \ref{DHR-lemma} the estimates%
\begin{eqnarray}
2^{v\alpha \left( y\right) }\left\vert \varphi _{v}\ast f(y)\right\vert
&\leq &C_{1}\text{ }2^{v\alpha \left( y\right) }\left( \eta _{v,w}\ast
|\varphi _{v}\ast f|^{t}(y)\right) ^{\frac{1}{t}}  \notag \\
&\leq &C_{2}\text{ }\left( \eta _{v,w-c_{\log }(\alpha )}\ast (2^{v\alpha
\left( \cdot \right) }|\varphi _{v}\ast f|)^{t}(y)\right) ^{\frac{1}{t}}
\label{esti-conv}
\end{eqnarray}%
are true for any $y\in \mathbb{R}^{n},v\in \mathbb{N}_{0}$ and any $w>n$.
Now divide both sides of \eqref{esti-conv} by $\left( 1+2^{v}\left\vert
x-y\right\vert \right) ^{a}$, in the right-hand side we use the inequality%
\begin{equation*}
\left( 1+2^{v}\left\vert x-y\right\vert \right) ^{-a}\leq \left(
1+2^{v}\left\vert x-z\right\vert \right) ^{-a}\left( 1+2^{v}\left\vert
y-z\right\vert \right) ^{a},\quad x,y,z\in \mathbb{R}^{n},
\end{equation*}%
in the left-hand side take the supremum over $y\in \mathbb{R}^{n}$ and get
for all $f\in \mathfrak{B}_{p(\cdot ),q(\cdot )}^{\alpha (\cdot ),\tau
(\cdot )}$, any $x\in P$ any $v\geq v_{P}^{+}$ and any $w>\max (n,at+c_{\log
}(\alpha ))$%
\begin{equation*}
\left( \varphi _{v}^{\ast ,a}2^{v\alpha \left( \cdot \right) }f(x)\right)
^{t}\leq C_{2}\text{ }\eta _{v,at}\ast (2^{v\alpha \left( \cdot \right)
}|\varphi _{v}\ast f|)^{t}(x)
\end{equation*}%
where $C_{2}>0$ is independent of $x,v$ and $f$. An application of Lemma \ref%
{Alm-Hastolemma1} gives that the left hand side of \eqref{estimate-B} is
bounded by 
\begin{eqnarray*}
&&C\sup_{P\in \mathcal{Q}}\Big\|\Big(\frac{\eta _{v,at}\ast (2^{v\alpha
\left( \cdot \right) }|\varphi _{v}\ast f|)^{t}}{\left\vert P\right\vert
^{\tau (\cdot )t}}\chi _{P}\Big)_{v\geq v_{P}^{+}}\Big\|_{\ell ^{\frac{%
q(\cdot )}{t}}(L^{\frac{p(\cdot )}{t}})}^{\frac{1}{t}} \\
&\leq &C\big\|\left( 2^{v\alpha \left( \cdot \right) }\varphi _{v}\ast
f\right) _{v}\big\|_{\ell ^{q(\cdot )}(L_{p(\cdot )}^{\tau (\cdot )})} \\
&=&C\big\|f\big\|_{\mathfrak{B}_{p(\cdot ),q(\cdot )}^{\alpha (\cdot ),\tau
(\cdot )}}.
\end{eqnarray*}%
\noindent The proof is complete.
\end{proof}

Atoms are the building blocks for the atomic decomposition.

\begin{definition}
\label{Atom-Def}Let $K\in \mathbb{N}_{0},L+1\in \mathbb{N}_{0}$ and let $%
\gamma >1$. A $K$-times continuous differentiable function $a\in C^{K}(%
\mathbb{R}^{n})$ is called $[K,L]$-atom centered at $Q_{v,m}$, $v\in \mathbb{%
N}_{0}$ and $m\in \mathbb{Z}^{n}$, if

\begin{equation}
\mathrm{supp}\text{ }a\subseteq \gamma Q_{v,m}  \label{supp-cond}
\end{equation}

\begin{equation}
|\partial ^{\beta }a(x)|\leq 2^{v(|\beta |+1/2)}\text{,\quad for\quad }0\leq
|\beta |\leq K,x\in \mathbb{R}^{n}  \label{diff-cond}
\end{equation}%
and if%
\begin{equation}
\int_{\mathbb{R}^{n}}x^{\beta }a(x)dx=0,\text{\quad for\quad }0\leq |\beta
|\leq L\text{ and }v\geq 1.  \label{mom-cond}
\end{equation}
\end{definition}

If the atom $a$ located at $Q_{v,m}$, that means if it fulfills %
\eqref{supp-cond}, then we will denote it by $a_{v,m}$. For $v=0$ or $L=-1$
there are no moment\ conditions\ \eqref{mom-cond} required.\vskip5pt

For proving the decomposition by atoms we need the following lemma, see
Frazier \& Jawerth \cite[Lemma 3.3]{FJ86}.

\begin{lemma}
\label{FJ-lemma}Let $\Phi $ and $\varphi $ satisfy, respectively, %
\eqref{Ass1} and \eqref{Ass2} and let $\varrho _{v,m}$ be an $\left[ K,L%
\right] $-atom. Then%
\begin{equation*}
\left\vert \varphi _{j}\ast \varrho _{v,m}(x)\right\vert \leq c\text{ }%
2^{(v-j)K+vn/2}\left( 1+2^{v}\left\vert x-x_{Q_{v,m}}\right\vert \right)
^{-M}
\end{equation*}%
if $v\leq j$, and%
\begin{equation*}
\left\vert \varphi _{j}\ast \varrho _{v,m}(x)\right\vert \leq c\text{ }%
2^{(j-v)(L+n+1)+vn/2}\left( 1+2^{j}\left\vert x-x_{Q_{v,m}}\right\vert
\right) ^{-M}
\end{equation*}%
if $v\geq j$, where $M$ is sufficiently large, $\varphi _{j}=2^{jn}\varphi
(2^{j}\cdot )$ and $\varphi _{0}$ is replaced by $\Phi $.
\end{lemma}

Now we come to the atomic decomposition theorem.

\begin{theorem}
\label{atomic-dec}\textit{Let }$\alpha ,\tau \in C_{\mathrm{loc}}^{\log
},\tau ^{-}>0$ \textit{and }$p,q\in \mathcal{P}_{0}^{\log }$ with $%
0<q^{-}\leq q^{+}<\infty $\textit{. }Let $0<p^{-}\leq p^{+}<\infty $ and let 
$K,L+1\in \mathbb{N}_{0}$ such that%
\begin{equation}
K\geq ([\alpha ^{+}+n\tau ^{+}]+1)^{+},  \label{K,L,B-F-cond}
\end{equation}%
and%
\begin{equation}
L\geq \max (-1,[n(\frac{1}{\min (1,\frac{(\tau p)^{-}}{\tau ^{+}})}%
-1)-\alpha ^{-}]).  \label{K,L,B-cond}
\end{equation}%
Then $f\in \mathcal{S}^{\prime }(\mathbb{R}^{n})$ belongs to $\mathfrak{B}%
_{p(\cdot ),q(\cdot )}^{\alpha (\cdot ),\tau (\cdot )}$, if and only if it
can be represented as%
\begin{equation}
f=\sum\limits_{v=0}^{\infty }\sum\limits_{m\in \mathbb{Z}^{n}}\lambda
_{v,m}\varrho _{v,m},\text{ \ \ \ \ converging in }\mathcal{S}^{\prime }(%
\mathbb{R}^{n})\text{,}  \label{new-rep}
\end{equation}%
where $\varrho _{v,m}$ are $\left[ K,L\right] $-atoms and $\lambda
=\{\lambda _{v,m}\in \mathbb{C}:v\in \mathbb{N}_{0},m\in \mathbb{Z}^{n}\}\in 
\mathfrak{b}_{p(\cdot ),q(\cdot )}^{\alpha (\cdot ),\tau (\cdot )}$.
Furthermore, $\mathrm{inf}\big\|\lambda \big\|_{\mathfrak{b}_{p(\cdot
),q(\cdot )}^{\alpha (\cdot ),\tau (\cdot )}}$, where the infimum is taken
over admissible representations\ \eqref{new-rep}\textrm{, }is an equivalent
quasi-norm in $\mathfrak{B}_{p(\cdot ),q(\cdot )}^{\alpha (\cdot ),\tau
(\cdot )}$.
\end{theorem}

The convergence in $\mathcal{S}^{\prime }(\mathbb{R}^{n})$ can be obtained
as a by-product of the proof using the same method as in \cite[Theorem 4.3]%
{D6}. If $p$, $q$, $\tau $, and $\alpha $ are constants, then the
restriction \eqref{K,L,B-F-cond}, and their counterparts, in the atomic
decomposition theorem are $K\geq ([\alpha +n\tau ]+1)^{+}$\ and $L\geq \max
(-1,[n(\frac{1}{\min (1,p)}-1)-\alpha ])$, which are essentially the
restrictions from the works of \cite[Theorem 3.12]{D4}.

\begin{proof}
{The proof follows the ideas in \cite[Theorem 6]{FJ86} and \cite{D6}.\vskip%
5pt }

\textit{Step 1}. Assume that $f\in \mathfrak{B}_{p(\cdot ),q(\cdot
)}^{\alpha (\cdot ),\tau (\cdot )}$ and let $\Phi $ and $\varphi $ satisfy,
respectively \eqref{Ass1} and \eqref{Ass2}. There exist \ functions $\Psi
\in \mathcal{S}(\mathbb{R}^{n})$ satisfying \eqref{Ass1} and $\psi \in 
\mathcal{S}(\mathbb{R}^{n})$ satisfying \eqref{Ass2} such that for all $\xi
\in \mathbb{R}^{n}$%
\begin{equation*}
f=\Psi \ast \widetilde{\Phi }\ast f+\sum_{v=1}^{\infty }\psi _{v}\ast 
\widetilde{\varphi }_{v}\ast f,
\end{equation*}%
see Section 3. Using the definition of the cubes $Q_{v,m}$ we obtain%
\begin{equation*}
f(x)=\sum\limits_{m\in \mathbb{Z}^{n}}\int_{Q_{0,m}}\widetilde{\Phi }%
(x-y)\Psi \ast f(y)dy+\sum_{v=1}^{\infty }2^{vn}\sum\limits_{m\in \mathbb{Z}%
^{n}}\int_{Q_{v,m}}\widetilde{\varphi }(2^{v}(x-y))\psi _{v}\ast f(y)dy,
\end{equation*}%
with convergence in $\mathcal{S}^{\prime }(\mathbb{R}^{n})$. We define for
every $v\in \mathbb{N}$ and all $m\in \mathbb{Z}^{n}$%
\begin{equation}
\lambda _{v,m}=C_{\theta }\sup_{y\in Q_{v,m}}\left\vert \psi _{v}\ast
f(y)\right\vert  \label{Coefficient}
\end{equation}%
where%
\begin{equation*}
C_{\theta }=\max \{\sup_{\left\vert y\right\vert \leq 1}\left\vert D^{\alpha
}\theta (y)\right\vert :\left\vert \alpha \right\vert \leq K\}.
\end{equation*}%
Define also%
\begin{equation}
\varrho _{v,m}(x)=\left\{ 
\begin{array}{ccc}
\frac{1}{\lambda _{v,m}}2^{vn}\int_{Q_{v,m}}\widetilde{\varphi }%
_{v}(2^{v}(x-y))\psi _{v}\ast f(y)dy & \text{if} & \lambda _{v,m}\neq 0 \\ 
0 & \text{if} & \lambda _{v,m}=0%
\end{array}%
.\right.  \label{K-L-atom}
\end{equation}%
Similarly we define for every $m\in \mathbb{Z}^{n}$ the numbers $\lambda
_{0,m}$ and the functions $\varrho _{0,m}$ taking in \eqref{Coefficient} and %
\eqref{K-L-atom} $v=0$ and replacing $\psi _{v}$ and $\widetilde{\varphi }$
by $\Psi $ and $\widetilde{\Phi }$, respectively. Let us now check that such 
$\varrho _{vm}$ are atoms in the sense of Definition \ref{Atom-Def}. Note
that the support and moment conditions are clear by \eqref{Ass1} and %
\eqref{Ass2}, respectively. It thus remains to check \eqref{diff-cond} in
Definition \ref{Atom-Def}. We have%
\begin{eqnarray*}
\left\vert D^{\beta }\varrho _{v,m}(x)\right\vert &\leq &\frac{%
2^{v(n+\left\vert \beta \right\vert )}}{C_{\theta }}\int_{Q_{v,m}}\left\vert
(D^{\beta }\widetilde{\varphi })(2^{v}(x-y))\right\vert \left\vert \psi
_{v}\ast f(y)\right\vert dy\Big(\sup_{y\in Q_{v,m}}\left\vert \psi _{v}\ast
f(y)\right\vert \Big)^{-1} \\
&\leq &\frac{2^{v(n+\left\vert \beta \right\vert )}}{C_{\theta }}%
\int_{Q_{v,m}}\left\vert (D^{\beta }\widetilde{\varphi })(2^{v}(x-y))\right%
\vert dy \\
&\leq &2^{v(n+\left\vert \beta \right\vert )}\left\vert Q_{v,m}\right\vert \\
&\leq &2^{v\left\vert \beta \right\vert }.
\end{eqnarray*}%
The modifications for the terms with $v=0$ are obvious.\vskip5pt

\textit{Step }2\textit{.} Next we show that there is a constant $c>0$ such
that $\big\|\lambda \big\|_{\mathfrak{b}_{p(\cdot ),q(\cdot )}^{\alpha
(\cdot ),\tau (\cdot )}}\leq c\big\|f\big\|_{\mathfrak{B}_{p(\cdot ),q(\cdot
)}^{\alpha (\cdot ),\tau (\cdot )}}$. For that reason we exploit the
equivalent quasi-norms given in Theorem \ref{fun-char} involving Peetre's
maximal function. Let $v\in \mathbb{N}$. Taking into account that $%
\left\vert x-y\right\vert \leq c$ $2^{-v}$ for $x,y\in Q_{v,m}$ we obtain%
\begin{equation*}
2^{v(\alpha \left( x\right) -\alpha \left( y\right) )}\leq \frac{c_{\log
}(\alpha )v}{\log (e+\frac{1}{\left\vert x-y\right\vert })}\leq \frac{%
c_{\log }(\alpha )v}{\log (e+\frac{2^{v}}{c})}\leq c
\end{equation*}%
if $v\geq \left[ \log _{2}c\right] +2$. If $0<v<\left[ \log _{2}c\right] +2$%
, then $2^{v(\alpha \left( x\right) -\alpha \left( y\right) )}\leq
2^{v(\alpha ^{+}-\alpha ^{-})}\leq c$. Therefore,%
\begin{equation*}
2^{v\alpha \left( x\right) }\left\vert \psi _{v}\ast f(y)\right\vert \leq c%
\text{ }2^{v\alpha \left( y\right) }\left\vert \psi _{v}\ast f(y)\right\vert
\end{equation*}%
for any $x,y\in Q_{v,m}$ and any $v\in \mathbb{N}$. Hence,%
\begin{eqnarray*}
\sum\limits_{m\in \mathbb{Z}^{n}}\lambda _{v,m}2^{v\alpha \left( x\right)
}\chi _{v,m}(x) &=&C_{\theta }\sum\limits_{m\in \mathbb{Z}^{n}}2^{v\alpha
\left( x\right) }\sup_{y\in Q_{v,m}}\left\vert \psi _{v}\ast f(y)\right\vert
\chi _{v,m}(x) \\
&\leq &c\sum\limits_{m\in \mathbb{Z}^{n}}\sup_{\left\vert z\right\vert \leq c%
\text{ }2^{-v}}\frac{2^{v\alpha (x-z)}\left\vert \psi _{v}\ast
f(x-z)\right\vert }{(1+2^{v}\left\vert z\right\vert )^{a}}(1+2^{v}\left\vert
z\right\vert )^{a}\chi _{v,m}(x) \\
&\leq &c\text{ }\psi _{v}^{\ast ,a}2^{v\alpha \left( \cdot \right)
}f(x)\sum\limits_{m\in \mathbb{Z}^{n}}\chi _{v,m}(x) \\
&=&c\text{ }\psi _{v}^{\ast ,a}2^{v\alpha \left( \cdot \right) }f(x),
\end{eqnarray*}%
where $a>\frac{m\tau ^{+}}{(\tau p)^{-}}$ and we have used $%
\sum\limits_{m\in \mathbb{Z}^{n}}\chi _{v,m}(x)=1$. This estimate and its
counterpart for $v=0$ (which can be obtained by a similar calculation) give%
\begin{equation*}
\left\Vert \lambda \right\Vert _{\mathfrak{b}_{p(\cdot ),q(\cdot )}^{\alpha
(\cdot ),\tau (\cdot )}}\leq c\left\Vert \left( \psi _{v}^{\ast
,a}2^{v\alpha \left( \cdot \right) }f\right) _{v}\right\Vert _{\ell
^{q(\cdot )}(L_{p(\cdot )}^{\tau (\cdot )})}\leq c\left\Vert f\right\Vert _{%
\mathfrak{B}_{p(\cdot ),q(\cdot )}^{\alpha (\cdot ),\tau (\cdot )}},
\end{equation*}%
by Theorem \ref{fun-char}.\vskip5pt

\textit{Step }3\textit{.} Assume that $f$ can be represented by %
\eqref{new-rep}, with $K$ and $L$ satisfying \eqref{K,L,B-F-cond} and %
\eqref{K,L,B-cond}, respectively. Similar arguments of {\cite[Theorem 4.3]%
{D6}}, by using Lemmas \ref{Alm-Hastolemma1}, \ref{Key-lemma}, show that $%
f\in \mathfrak{B}_{p\left( \cdot \right) ,q\left( \cdot \right) }^{\alpha
\left( \cdot \right) ,\tau \left( \cdot \right) }$ and that for some $c>0$, $%
\left\Vert f\right\Vert _{\mathfrak{B}_{p\left( \cdot \right) ,q\left( \cdot
\right) }^{\alpha \left( \cdot \right) ,\tau \left( \cdot \right) }}\leq
c\left\Vert \lambda \right\Vert _{\mathfrak{b}_{p\left( \cdot \right)
,q\left( \cdot \right) }^{\alpha \left( \cdot \right) ,\tau \left( \cdot
\right) }}$.
\end{proof}

\section{Appendix}

Here we present more technical proofs of the Lemmas.\vskip5pt

{\emph{Proof of Lemma} \ref{Alm-Hastolemma1}.} By the scaling argument, we
see that it suffices to consider when 
\begin{equation}
\left\Vert \left( f_{v}\right) _{v}\right\Vert _{\ell ^{q(\cdot
)}(L_{p(\cdot )}^{\tau (\cdot )})}\leq 1  \label{modular2}
\end{equation}
and show that for any dyadic cube $P$ 
\begin{equation*}
\sum_{v=v_{P}^{+}}^{\infty }\Big\|\Big|\frac{c\text{ }\eta _{v,m}\ast
\left\vert f_{v}\right\vert }{|P|^{\tau (\cdot )}}\Big|^{q(\cdot )}\chi _{P}%
\Big\|_{\frac{p(\cdot )}{q(\cdot )}}\leq 1
\end{equation*}%
for some constant $c>0$. We distinguish two cases:

\noindent \textbf{Case 1.} $l(P)>1$. Let $Q_{v}\subset P$ be a cube, with $%
\ell \left( Q_{v}\right) =2^{-v}$ and $x\in Q_{v}\subset P$. We have%
\begin{eqnarray*}
&&\eta _{v,m}\ast |f_{v}|(x) \\
&=&2^{vn}\int_{\mathbb{R}^{n}}\frac{\left\vert f_{v}(z)\right\vert }{\left(
1+2^{v}\left\vert x-z\right\vert \right) ^{m}}dz \\
&=&\int_{3Q_{v}}\cdot \cdot \cdot dz+\sum_{k\in \mathbb{Z}^{n},\left\Vert
k\right\Vert _{\infty }\geq 2}\int_{Q_{v}^{k}}\cdot \cdot \cdot dz \\
&=&J_{v}^{1}(f_{v}\chi _{3Q_{v}})(x)+\sum_{k\in \mathbb{Z}^{n},\left\Vert
k\right\Vert _{\infty }\geq 2}J_{v,k}^{2}(f_{v}\chi _{Q_{v}^{k}})(x),
\end{eqnarray*}%
where $Q_{v}^{k}=Q_{v}+kl(Q_{v})$. Let $0<r<\frac{1}{2}\min (p^{-},q^{-},2)$
and define $\tilde{p}=\frac{p}{r}$ and $\tilde{q}=\frac{q}{r}$. Then
clearly, $\frac{1}{\tilde{p}}+\frac{1}{\tilde{q}}\leq 1$. Thus we obtain%
\begin{eqnarray}
&&\Big\|\Big(\frac{\eta _{v,m}\ast \left\vert f_{v}\right\vert }{|P|^{\tau
(\cdot )}}\chi _{P}\Big)_{v\in \mathbb{N}_{0}}\Big\|_{L^{p(\cdot )}(\ell
^{q(\cdot )})}^{r}  \notag \\
&\leq &\Big\|\Big(\frac{J_{v}^{1}(f_{v}\chi _{3Q_{v}})}{|P|^{\tau (\cdot )}}%
\chi _{P}\Big)_{v\in \mathbb{N}_{0}}\Big\|_{L^{p(\cdot )}(\ell ^{q(\cdot
)})}^{r}  \label{sum2} \\
&&+\sum_{k\in \mathbb{Z}^{n},\left\Vert k\right\Vert _{\infty }\geq 2}\Big\|%
\Big(\frac{J_{v,k}^{2}(f_{v}\chi _{Q_{v}^{k}})}{|P|^{\tau (\cdot )}}\chi _{P}%
\Big)_{v\in \mathbb{N}_{0}}\Big\|_{L^{p(\cdot )}(\ell ^{q(\cdot )})}^{r}.
\label{sum1}
\end{eqnarray}%
\textit{Estimate of }\eqref{sum2}. We will prove that \eqref{sum2} is
bounded\ by a constant independent of $P$. Clearly, we need to show that 
\begin{eqnarray*}
\Big\|\Big|\frac{c\text{ }J_{v}^{1}(f_{v}\chi _{3Q_{v}})}{|P|^{\tau (\cdot )}%
}\Big|^{q(\cdot )}\chi _{P}\Big\|_{\frac{p(\cdot )}{q(\cdot )}} &\leq &\Big\|%
\Big|\frac{f_{v}}{|P|^{\tau (\cdot )}}\Big|^{q(\cdot )}\chi _{3P}\Big\|_{%
\frac{p(\cdot )}{q(\cdot )}}+2^{-v} \\
&=&\delta
\end{eqnarray*}%
for some positive constant $c$. This claim can be reformulated as showing
that%
\begin{equation}
\Big\|\delta ^{-\frac{1}{q\left( \cdot \right) }}\frac{c\text{ }%
J_{v}^{1}(f_{v}\chi _{3Q_{v}})}{|P|^{\tau (\cdot )}}\chi _{P}\Big\|_{p(\cdot
)}\leq 1.  \label{est-Jnegative}
\end{equation}%
Let $d>0$ be such that $\tau ^{+}<d<\left( \tau p\right) ^{-}$. We have%
\begin{equation*}
\frac{M_{3Q_{v}}\left( f_{v}\right) }{|P|^{\tau (\cdot )}}=\Big(\frac{\big(%
M_{3Q_{v}}\left( f_{v}\right) \big)^{\frac{d}{\tau (\cdot )}}}{|P|^{d}}\Big)%
^{\frac{\tau (\cdot )}{d}}.
\end{equation*}%
Hence, we will prove that%
\begin{equation*}
\Big\|\frac{\delta ^{-\frac{d}{q\left( \cdot \right) \tau (\cdot )}}\big(%
M_{3Q_{v}}\left( f_{v}\right) \left( \cdot \right) \big)^{\frac{d}{\tau
(\cdot )}}}{|P|^{d}}\chi _{P}\Big\|_{\frac{p(\cdot )\tau (\cdot )}{d}%
}\lesssim 1.
\end{equation*}%
By H\"{o}lder's inequality, 
\begin{equation*}
|Q_{v}|M_{3Q_{v}}\big(\left\vert f_{v}\right\vert ^{\frac{d}{\tau (\cdot )}}%
\big)\lesssim \Big\|\frac{\left\vert f_{v}\right\vert ^{\frac{1}{\tau (\cdot
)}}}{|3Q_{v}|}\chi _{3Q_{v}}\Big\|_{p(\cdot )\tau (\cdot )}^{d}\big\|%
|3Q_{v}|\chi _{3Q_{v}}\big\|_{t(\cdot )}^{d},
\end{equation*}%
where $\frac{1}{d}=\frac{1}{p(\cdot )\tau (\cdot )}+\frac{1}{t(\cdot )}$.
The second quasi-norm is bounded, while the first is bounded if and only if%
\begin{equation*}
\Big\|\frac{f_{v}}{|Q_{v}|^{\tau (\cdot )}}\chi _{3Q_{v}}\Big\|_{p(\cdot
)}\lesssim 1.
\end{equation*}%
{Notice that }$3Q_{v}{\subset }${$\cup _{h=1}^{3^{n}}Q_{v}^{h}$, where $%
\{Q_{v}^{h}\}_{h=1}^{3^{n}}$ are disjoint dyadic cubes with side length $%
l(Q_{v}^{h})=l(Q_{v})$. Therefore $\chi _{3Q_{v}}\leq \sum_{h=1}^{3^{n}}\chi
_{Q_{v}^{h}}$ and 
\begin{equation*}
\Big\|\frac{f_{v}}{|Q_{v}|^{\tau (\cdot )}}\chi _{3Q_{v}}\Big\|_{p(\cdot
)}\leq c\sum_{h=1}^{3^{n}}\Big\|\frac{f_{v}}{|Q_{v}^{h}|^{\tau \left( \cdot
\right) }}\chi _{Q_{v}^{h}}\Big\|_{p(\cdot )}\lesssim 1,
\end{equation*}%
where we used }\eqref{modular2}. We can use Lemma \ref{DHHR-estimate} to
obtain that%
\begin{equation*}
\left( \beta M_{3Q_{v}}\left( f_{v}\right) \right) ^{\frac{d}{\tau (x)}}
\end{equation*}%
can be estimated by 
\begin{equation*}
M_{3Q_{v}}\big(\big|f_{v}\big|^{\frac{d}{\tau (\cdot )}}\big)+\left\vert
Q_{v}\right\vert ^{s}g(x)
\end{equation*}%
for any $x\in Q_{v}$ and any $s>0$, where $g$ is the same function as in
Theorem \ref{Sobolev-emb}. Taking into account that $\frac{1}{q}$ and $\tau $
are log-H\"{o}lder continuous, $\delta \in \lbrack 2^{-v},1+2^{-v}]$, by
Lemma \ref{DHR-lemma};%
\begin{equation*}
\delta ^{-\frac{d}{q\left( x\right) \tau (x)}}\left( \beta M_{3Q_{v}}\left(
f_{v}\right) \right) ^{\frac{d}{\tau (x)}}
\end{equation*}%
does not exceed%
\begin{equation*}
M_{3Q_{v}}\big(\big|\delta ^{-\frac{1}{q\left( \cdot \right) }}f_{v}\big|^{%
\frac{d}{\tau (\cdot )}}\big)+2^{\frac{vd}{q\left( x\right) \tau (x)}%
}\left\vert Q_{v}\right\vert ^{s}g(x)\lesssim M_{3Q_{v}}\big(\big|\delta ^{-%
\frac{1}{q\left( \cdot \right) }}f_{v}\big|^{\frac{d}{\tau \left( \cdot
\right) }}\big)+h(x),
\end{equation*}%
where we used $\max_{x\in Q_{v}}2^{vd/q\left( x\right) \tau (x)}\left\vert
Q_{v}\right\vert ^{s}\leq 1$, since $s>0$ can be taken large enough, where $%
h $ is the same function as in Theorem \ref{Sobolev-emb}. Therefore, 
\begin{eqnarray*}
\Big\|\frac{\Big(\delta ^{-\frac{1}{q\left( \cdot \right) }}M_{3Q_{v}}\left(
f_{v}\right) \Big)^{\frac{d}{\tau (\cdot )}}}{|P|^{d}}\Big\|_{\frac{p(\cdot
)\tau (\cdot )}{d}} &\lesssim &\Big\|\mathcal{M}\big(\frac{\delta ^{-\frac{d%
}{q\left( \cdot \right) \tau \left( \cdot \right) }}\left\vert
f_{v}\right\vert ^{\frac{d}{\tau (\cdot )}}}{|P|^{d}}\chi _{3Q_{v}}\big)%
\Big\|_{\frac{p(\cdot )\tau (\cdot )}{d}}+c \\
&\lesssim &\Big\|\frac{\delta ^{-\frac{d}{q\left( \cdot \right) \tau \left(
\cdot \right) }}\left\vert f_{v}\right\vert ^{\frac{d}{\tau (\cdot )}}}{%
|P|^{d}}\chi _{3P}\Big\|_{\frac{p(\cdot )\tau (\cdot )}{d}}+c,
\end{eqnarray*}%
since $\frac{p\tau }{d}\in \mathcal{P}^{\mathrm{log}}$, $(\frac{p\tau }{d}%
)^{-}>1$ and $\mathcal{M}:L^{\frac{p(\cdot )\tau (\cdot )}{d}}\rightarrow L^{%
\frac{p(\cdot )\tau (\cdot )}{d}}$ is bounded. The last norm is bounded if
and only if 
\begin{equation*}
\Big\|\frac{\delta ^{-\frac{1}{q\left( \cdot \right) }}\left\vert
f_{v}\right\vert \chi _{3P}}{|P|^{\tau (\cdot )}}\Big\|_{p(\cdot )}\lesssim
1,
\end{equation*}%
which follows immediately from the definition of $\delta $.\ 

\textit{Estimate of }\eqref{sum1}. We will prove that \eqref{sum1} is
bounded\ by a constant independent of $P$. The summation in \eqref{sum1} can
be rewritten as%
\begin{equation*}
\sum_{k\in \mathbb{Z},|k|\leq 4\sqrt{n}}\cdot \cdot \cdot +\sum_{k\in 
\mathbb{Z}^{n},|k|>4\sqrt{n}}\cdot \cdot \cdot .
\end{equation*}%
The estimation of the first sum follows in the same manner as in the
estimate of $J_{v}^{1}(f_{v})$, so we need only to estimate the second sum.
Let now prove that%
\begin{equation*}
\Big\|\Big(|k|^{b(\cdot )}\frac{J_{v,k}^{2}(f_{v}\chi _{Q_{v}^{k}})}{%
|P|^{\tau (\cdot )}}\chi _{P}\Big)_{v\in \mathbb{N}_{0}}\Big\|_{L^{p(\cdot
)}(\ell ^{q(\cdot )})}\lesssim \Big\|\Big(\frac{f_{v}}{|\tilde{Q}^{k}|^{\tau
(\cdot )}}\chi _{\tilde{Q}^{k}}\Big)_{v\in \mathbb{N}_{0}}\Big\|_{L^{p(\cdot
)}(\ell ^{q(\cdot )})},
\end{equation*}%
where $\tilde{Q}^{k}=Q(c_{P},2\left\vert k\right\vert l(P))$ and%
\begin{equation*}
b(\cdot )=m-n(1+\frac{1}{t^{-}})\tau ^{+}-2\frac{c_{\log }\big(\frac{d}{%
q\tau }\big)\tau (\cdot )}{d}-\frac{s\tau (\cdot )}{d}
\end{equation*}%
and $s$ will be chosen later. Again, by the scaling argument, we see that it
suffices to consider when the last norm is less than or equal $1$ and show
that the modular of a constant times the function on the left-hand side is
bounded. In particular, we will show that for any dyadic cube $P$ 
\begin{equation*}
\sum_{v=0}^{\infty }\Big\|\Big|\frac{c\text{ }|k|^{b(\cdot
)}J_{v,k}^{2}(f_{v}\chi _{Q_{v}^{k}})}{|P|^{\tau (\cdot )}}\Big|^{q(\cdot
)}\chi _{P}\Big\|_{\frac{p(\cdot )}{q(\cdot )}}\leq 1
\end{equation*}%
for some positive constant $c$. This estimate follows from the inequality%
\begin{eqnarray*}
\Big\|\Big|\frac{c\text{ }|k|^{b(\cdot )}J_{v,k}^{2}(f_{v}\chi _{Q_{v}^{k}})%
}{|P|^{\tau (\cdot )}}\Big|^{q(\cdot )}\chi _{P}\Big\|_{\frac{p(\cdot )}{%
q(\cdot )}} &\leq &\Big\|\Big|\frac{f_{v}}{|\tilde{Q}^{k}|^{\tau (\cdot )}}%
\Big|^{q(\cdot )}\chi _{\tilde{Q}^{k}}\Big\|_{\frac{p(\cdot )}{q(\cdot )}%
}+2^{-v} \\
&=&\delta
\end{eqnarray*}%
for any $v\in \mathbb{N}_{0}$. This claim can be reformulated as showing that%
\begin{equation}
\Big\|\delta ^{-\frac{d}{q\left( \cdot \right) \tau (\cdot )}}\frac{\big(%
|k|^{b(\cdot )}J_{v,k}^{2}(f_{v}\chi _{Q_{v}^{k}})\big)^{\frac{d}{\tau
(\cdot )}}}{|P|^{d}}\chi _{P}\Big\|_{\frac{p(\cdot )\tau (\cdot )}{d}%
}\lesssim 1.  \label{est-J2}
\end{equation}%
Let $z\in Q_{v}^{k},$ $x\in Q_{v}$ with $k\in \mathbb{Z}^{n}$ and $|k|>4%
\sqrt{n}$. Then $z=h+k2^{-v}$ for some $h\in Q_{v}$, $\left\vert
x-z\right\vert \geq \left\vert k\right\vert 2^{-v-1}$. Hence 
\begin{eqnarray*}
\delta ^{-\frac{1}{q\left( x\right) }}|k|^{b\left( x\right)
}J_{v,k}^{2}(f_{v}\chi _{Q_{v}^{k}}) &\lesssim &\delta ^{-\frac{1}{q\left(
x\right) }}|k|^{b\left( x\right) -m}M_{Q_{v}^{k}}\left( f_{v}\right) \\
&\lesssim &\delta ^{-\frac{1}{q\left( x\right) }}|k|^{b\left( x\right)
-m+n(1+\frac{1}{t^{-}})\tau ^{+}}M_{Q_{v}^{k}}(\left\vert k\right\vert
^{-n(1+\frac{1}{t\left( \cdot \right) })\tau \left( \cdot \right) }f_{v}) \\
&\lesssim &\delta ^{-\frac{1}{q\left( x\right) }}|k|^{-(2c_{\log }(\frac{d}{%
q\tau })+s)\frac{s\tau (\cdot )}{d}}M_{Q_{v}^{k}}(\left\vert k\right\vert
^{-n(1+\frac{1}{t\left( \cdot \right) })\tau \left( \cdot \right) }f_{v})
\end{eqnarray*}%
for any $x\in Q_{v}$ and any $v\in \mathbb{N}_{0}$, where $\frac{1}{d}=\frac{%
1}{p(\cdot )\tau (\cdot )}+\frac{1}{t(\cdot )}$. Observe that $%
Q_{v}^{k}\subset Q(x,\left\vert k\right\vert 2^{1-v})=\tilde{Q}_{v}^{k}$. By
H\"{o}lder's inequality, 
\begin{equation*}
|\tilde{Q}_{v}^{k}|M_{\tilde{Q}_{v}^{k}}\big(\left\vert k\right\vert ^{-n(1+%
\frac{1}{t\left( \cdot \right) })d}\left\vert f_{v}\right\vert ^{\frac{d}{%
\tau (\cdot )}}\big)\lesssim \Big\|\frac{\left\vert f_{v}\right\vert ^{\frac{%
1}{\tau (\cdot )}}}{|\tilde{Q}_{v}^{k}|}\chi _{\tilde{Q}_{v}^{k}}\Big\|%
_{p(\cdot )\tau (\cdot )}^{d}\big\||\tilde{Q}_{v}^{k}|\left\vert
k\right\vert ^{-n(1+\frac{1}{t(\cdot )})}\chi _{\tilde{Q}_{v}^{k}}\big\|%
_{t(\cdot )}^{d}.
\end{equation*}%
The second quasi-norm is bounded, while the first is bounded if and only if%
\begin{equation*}
\Big\|\frac{f_{v}}{|\tilde{Q}_{v}^{k}|^{\tau (\cdot )}}\chi _{\tilde{Q}%
_{v}^{k}}\Big\|_{p(\cdot )}\lesssim 1,\quad v\in \mathbb{N}_{0},
\end{equation*}%
{which follows by }\eqref{modular2}. Again by Lemma \ref{DHHR-estimate},%
\begin{equation*}
\Big(\beta M_{\tilde{Q}_{v}^{k}}\big(\left\vert k\right\vert ^{-n(1+\frac{1}{%
t\left( \cdot \right) })\tau \left( \cdot \right) }f_{v}\big)\Big)^{\frac{d}{%
\tau (x)}}
\end{equation*}%
does not exceed 
\begin{equation*}
M_{\tilde{Q}_{v}^{k}}\big(\big|\left\vert k\right\vert ^{-n(1+\frac{1}{%
t\left( \cdot \right) })\tau \left( \cdot \right) }f_{v}\big|^{\frac{d}{\tau
(\cdot )}}\big)+\min (1,\left\vert k\right\vert ^{ns}2^{(1-v)ns})\big(\left(
e+\left\vert x\right\vert \right) ^{-s}+M_{\tilde{Q}_{v}^{k}}\left( \left(
e+\left\vert y\right\vert \right) ^{-s}\right) \big)
\end{equation*}%
for any $s>0$ large enough. Hence, 
\begin{equation*}
\delta ^{-\frac{d}{q\left( x\right) \tau (x)}}\Big(\beta M_{\tilde{Q}%
_{v}^{k}}\big(\left\vert k\right\vert ^{-n(1+\frac{1}{t\left( \cdot \right) }%
)\tau \left( \cdot \right) }f_{v}\big)\Big)^{\frac{d}{\tau (x)}}
\end{equation*}%
is bounded by%
\begin{eqnarray*}
&&c\left\vert k\right\vert ^{2c_{\log }(\frac{d}{q\tau })}M_{\tilde{Q}%
_{v}^{k}}\big(\delta ^{-\frac{d}{q\left( \cdot \right) \tau (\cdot )}%
}\left\vert k\right\vert ^{-n(1+\frac{1}{t\left( \cdot \right) }%
)d}\left\vert f_{v}\right\vert ^{\frac{d}{\tau \left( \cdot \right) }}\big)%
+2^{\frac{vd}{\left( q\tau \right) ^{-}}}\min (1,\left\vert k\right\vert
^{ns}2^{(1-v)ns})h\left( x\right) \\
&\lesssim &\left\vert k\right\vert ^{2c_{\log }(\frac{d}{q\tau })}M_{\tilde{Q%
}_{v}^{k}}\big(\delta ^{-\frac{d}{q\left( \cdot \right) \tau (\cdot )}%
}\left\vert k\right\vert ^{-n(1+\frac{1}{t\left( \cdot \right) }%
)d}\left\vert f_{v}\right\vert ^{\frac{d}{\tau \left( \cdot \right) }}\big)%
+\left\vert k\right\vert ^{ns}h\left( x\right) ,
\end{eqnarray*}%
where $s>0$ large enough such that $s>\frac{d}{n\left( q\tau \right) ^{-}}$.
Therefore, the left-hand side of \eqref{est-J2} is bounded by 
\begin{eqnarray*}
&&\Big\|c\text{ }\mathcal{M}\Big(|k|^{-nd}\frac{\delta ^{-\frac{d}{q\left(
\cdot \right) \tau \left( \cdot \right) }}\left\vert f_{v}\right\vert ^{%
\frac{d}{\tau (\cdot )}}\chi _{\tilde{Q}^{k}}}{|P|^{d}}\Big)\Big\|_{\frac{%
p(\cdot )\tau (\cdot )}{d}}+C \\
&\lesssim &\Big\|\frac{\delta ^{-\frac{d}{q\left( \cdot \right) \tau \left(
\cdot \right) }}\left\vert f_{v}\right\vert ^{\frac{d}{\tau (\cdot )}}\chi _{%
\tilde{Q}^{k}}}{|\tilde{Q}^{k}|^{d}}\Big\|_{\frac{p(\cdot )\tau (\cdot )}{d}%
}+C,
\end{eqnarray*}%
after using the fact that $\mathcal{M}:L^{\frac{p(\cdot )\tau (\cdot )}{d}%
}\rightarrow L^{\frac{p(\cdot )\tau (\cdot )}{d}}$ is bounded. The last norm
is bounded if and only if 
\begin{equation*}
\Big\|\frac{\delta ^{-\frac{1}{q\left( \cdot \right) }}f_{v}\chi _{\tilde{Q}%
^{k}}}{|\tilde{Q}^{k}|^{\tau (\cdot )}}\Big\|_{p(\cdot )}\leq 1,
\end{equation*}%
which follows immediately from the definition of $\delta $. Since $m$ can be
taken large enough, then the second sum in \eqref{sum1} is bounded by%
\begin{eqnarray*}
\sum_{k\in \mathbb{Z}^{n},|k|>4\sqrt{n}}|k|^{-b^{-}r}\Big\|\Big(\frac{f_{v}}{%
|\tilde{Q}^{k}|^{\tau (\cdot )}}\chi _{\tilde{Q}^{k}}\Big)_{v\geq 2v_{P}^{+}}%
\Big\|_{L^{p(\cdot )}(\ell ^{q(\cdot )})}^{r} &\leq &\sum_{k\in \mathbb{Z}%
^{n},|k|>4\sqrt{n}}|k|^{-b^{-}r}\big\|(f_{v})_{v}\big\|_{\ell ^{q(\cdot
)}(L_{p(\cdot )}^{\tau (\cdot )})}^{r} \\
&\lesssim &1.
\end{eqnarray*}%
\textbf{Case 2. }$l(P)\leq 1$. As before, 
\begin{equation*}
\eta _{v,m}\ast \left\vert f_{v}\right\vert (x)\lesssim J_{v}^{1}(f_{v}\chi
_{3P})(x)+\sum_{k\in \mathbb{Z}^{n},\left\Vert k\right\Vert _{\infty }\geq
2}J_{v,k}^{2}(f_{v}\chi _{P+kl(P)})(x).
\end{equation*}%
We see that 
\begin{equation*}
J_{v}^{1}(f_{v}\chi _{3P})(x)=\eta _{v,m}\ast \left( \left\vert
f_{v}\right\vert \chi _{3P}\right) (x),\quad x\in P
\end{equation*}%
and {since }$\tau ${\ is log-H\"{o}lder continuous, we have }%
\begin{equation*}
\left\vert P\right\vert ^{-\tau (x)}\leq c\left\vert P\right\vert ^{-\tau
(y)}(1+2^{v_{P}}\left\vert x-y\right\vert )^{c_{\log }\left( \tau \right)
}\leq c\left\vert P\right\vert ^{-\tau (y)}
\end{equation*}%
for any $x\in P$ any $y\in 3P$ and any $v\geq v_{P}$. Hence%
\begin{equation*}
\left\vert P\right\vert ^{-\tau (x)}J_{v}^{1}(f_{v}\chi _{3P})(x)\lesssim
\eta _{v,m-c_{\log }\left( \tau \right) }\ast \big(\left\vert P\right\vert
^{-\tau (\cdot )}\left\vert f_{v}\right\vert \chi _{3P}\big)(x),\quad x\in P.
\end{equation*}%
Also, we have 
\begin{equation*}
\left\vert P\right\vert ^{-\tau (x)}J_{v,k}^{2}(f_{v}\chi
_{P+kl(P)})(x)\lesssim \eta _{v,m-c_{\log }\left( \tau \right) }\ast \big(%
\left\vert P\right\vert ^{-\tau (\cdot )}\left\vert f_{v}\right\vert \chi
_{P+kl(P)}\big)(x).
\end{equation*}%
As before, we obtain 
\begin{equation*}
\sum_{v=v_{P}}^{\infty }\Big\|\Big|\frac{c\text{ }\eta _{v,m}\ast f_{v}}{%
|P|^{\tau (\cdot )}}\Big|^{q(\cdot )}\chi _{P}\Big\|_{\frac{p(\cdot )}{%
q(\cdot )}}\leq 1,
\end{equation*}%
where we did not need to use Lemma \ref{DHHR-estimate}, which could be used
only to move $\left\vert P\right\vert ^{\tau (\cdot )}$ inside the
convolution. The proof is complete.\vskip5pt

{\emph{Proof of Lemma }\ref{key-estimate1}. We claim that}%
\begin{equation}
2^{-v\frac{n}{r}}|\omega _{v}\ast f(x)|\lesssim \big\|\omega _{v}\ast f\big\|%
_{\widetilde{L_{\tau (\cdot )}^{p(\cdot )}}}  \label{key-estimate}
\end{equation}%
{\ for any }$x\in \mathbb{R}^{n}$, any $0<r<p^{-}$ and any $v\in \mathbb{N}%
_{0}$. Indeed. By {Lemma \ref{r-trick}, we have 
\begin{equation*}
\left\vert \omega _{v}\ast f\left( x\right) \right\vert \leq c\text{ }(\eta
_{v,m}\ast \left\vert \omega _{v}\ast f\right\vert ^{r}(x))^{1/r},
\end{equation*}%
for any }$x\in \mathbb{R}^{n}$, any $m>n$, $0<r<p^{-}$ and any $v\in \mathbb{%
N}_{0}$. We write%
\begin{equation*}
\eta _{v,m}\ast \left\vert \omega _{v}\ast f\right\vert ^{r}(x)\lesssim
\sum_{i=0}^{\infty }2^{-i(m-n)}M_{B(x,2^{i-v})}(\left\vert \omega _{v}\ast
f\right\vert ^{r}),
\end{equation*}%
where the implicit constant independent of $x$ and $x$. H\"{o}lder's
inequality leads to%
\begin{eqnarray*}
M_{B(x,2^{i-v})}(\left\vert \omega _{v}\ast f\right\vert ^{r}) &\lesssim
&2^{(v-i)n}\big\|(\omega _{v}\ast f)\chi _{B(x,2^{i-v})}\big\|_{p(\cdot
)}^{r}\big\|\chi _{B(x,2^{i-v})}\big\|_{h(\cdot )}^{r} \\
&\lesssim &2^{(v-i)n+inr\tau ^{+}}\big\|\omega _{v}\ast f\big\|_{\widetilde{%
L_{\tau (\cdot )}^{p(\cdot )}}}^{r}\big\|\chi _{B(x,2^{i})}\big\|_{h(\cdot
)}^{r},
\end{eqnarray*}%
where $\frac{1}{r}=\frac{1}{p(\cdot )}+\frac{1}{h(\cdot )}$. Making $m$
large enough \eqref{key-estimate} follows.

Let $P$ be any dyadic cube. {We use again Lemma \ref{r-trick}, in the form%
\begin{equation*}
\left\vert \theta _{v}\ast \omega _{v}\ast f\left( x\right) \right\vert \leq
c\text{ }(\eta _{v,m}\ast |\omega _{v}\ast f|^{r}(x))^{1/r},
\end{equation*}%
where $0<r<\min (p^{-},\frac{(p\tau )^{-}}{\tau ^{+}})$, $m>n$ large enough
and $x\in P$. }By the scaling argument, we see that it suffices to prove that%
\begin{equation*}
\Big\|\frac{\theta _{v}\ast \omega _{v}\ast f}{|P|^{\tau (\cdot )}}\chi _{P}%
\Big\|_{p(\cdot )}\lesssim 1
\end{equation*}%
for any dyadic cube $P$, with $l(P)\geq 1$, whenever $\big\|\omega _{v}\ast f%
\big\|_{\widetilde{L_{\tau (\cdot )}^{p(\cdot )}}}\leq 1$. Let $Q_{v}\subset
P$ be a cube, with $l\left( Q_{v}\right) =2^{-v}$ and $x\in Q_{v}\subset P$.
As in Lemma {\ref{Alm-Hastolemma1},}%
\begin{equation*}
\eta _{v,m}\ast \left\vert \omega _{v}\ast f\right\vert ^{r}(x)\leq
J_{v}^{1}(\left\vert \omega _{v}\ast f\right\vert ^{r}\chi
_{3Q_{v}})(x)+\sum_{k\in \mathbb{Z}^{n},\left\Vert k\right\Vert _{\infty
}\geq 2}J_{v,k}^{2}(\left\vert \omega _{v}\ast f\right\vert ^{r}\chi
_{Q_{v}^{k}})(x).
\end{equation*}%
{Thus we obtain%
\begin{eqnarray}
\Big\|\frac{\theta _{v}\ast \omega _{v}\ast f}{|P|^{\tau (\cdot )}}\chi _{P}%
\Big\|_{p(\cdot )}^{r} &\lesssim &\Big\|\frac{J_{v}^{1}(\left\vert \omega
_{v}\ast f\right\vert ^{r}\chi _{3Q_{v}})}{|P|^{r\tau (\cdot )}}\chi _{P}%
\Big\|_{\frac{p(\cdot )}{r}}  \notag \\
&&+\sum_{k\in \mathbb{Z}^{n},\left\Vert k\right\Vert _{\infty }\geq 2}\Big\|%
\frac{J_{v,k}^{2}(\left\vert \omega _{v}\ast f\right\vert ^{r}\chi
_{Q_{v}^{k}})}{|P|^{r\tau (\cdot )}}\chi _{P}\Big\|_{\frac{p(\cdot )}{r}}.
\label{norm1}
\end{eqnarray}%
Let us prove that the first norm on the right-hand side is bounded. We have%
\begin{equation*}
|J_{v}^{1}(\left\vert \omega _{v}\ast f\right\vert ^{r}\chi
_{3Q_{v}})(x)|\lesssim M_{3Q_{v}}\left( \left\vert \omega _{v}\ast
f\right\vert ^{r}\right) \left( x\right) .
\end{equation*}%
}Let $d>0$ be such that $\tau ^{+}<d<\frac{(p\tau )^{-}}{r}$. We have%
\begin{equation*}
\frac{M_{3Q_{v}}\left( \left\vert \omega _{v}\ast f\right\vert ^{r}\right) }{%
|P|^{r\tau (\cdot )}}=\Big(2^{v\frac{nd}{\tau \left( \cdot \right) }}\frac{%
\big(M_{3Q_{v}}\big(2^{-vn}|\omega _{v}\ast f|^{r}\big)\big)^{\frac{d}{\tau
(\cdot )}}}{|P|^{dr}}\Big)^{\frac{\tau (\cdot )}{d}}.
\end{equation*}%
By \eqref{key-estimate}, Lemma \ref{DHHR-estimate} and the fact that $2^{-%
\frac{vnd}{\tau \left( x\right) }}\approx 2^{-\frac{vnd}{\tau \left(
y\right) }}$, $x,y\in 3Q_{v}$,%
\begin{equation*}
2^{v\frac{nd}{\tau \left( x\right) }}\big(\beta M_{3Q_{v}}\big(%
2^{-vnr}\left\vert \omega _{v}\ast f\right\vert ^{r}\big)\big)^{\frac{d}{%
\tau (x)}}\lesssim M_{3Q_{v}}\big(\left\vert \omega _{v}\ast f\right\vert ^{%
\frac{rd}{\tau \left( \cdot \right) }}\big)+2^{\frac{vnrd}{\tau ^{-}}%
}2^{-snv}h(x)
\end{equation*}%
for any $s>0$ large enought and any $x\in Q_{v}$, where the implicit
constant is independent of $x$ and $v$. Hence%
\begin{eqnarray*}
\Big\|\Big(\frac{J_{v}^{1}(\left\vert \omega _{v}\ast f\right\vert ^{r}\chi
_{3Q_{v}})}{|P|^{r\tau (\cdot )}}\chi _{P}\Big)^{\frac{d}{\tau (\cdot )}}%
\Big\|_{\frac{p(\cdot )\tau \left( \cdot \right) }{dr}} &\lesssim &\Big\|%
\mathcal{M}\big(\frac{\left\vert \omega _{v}\ast f\right\vert ^{\frac{dr}{%
\tau (\cdot )}}\chi _{3Q_{v}}}{|P|^{rd}}\big)\Big\|_{\frac{p(\cdot )\tau
\left( \cdot \right) }{dr}}+c \\
&\lesssim &\Big\|\frac{\left\vert \omega _{v}\ast f\right\vert ^{\frac{dr}{%
\tau (\cdot )}}}{|P|^{rd}}\chi _{3P}\Big\|_{\frac{p(\cdot )\tau \left( \cdot
\right) }{dr}}+c,
\end{eqnarray*}%
after using the fact that $\mathcal{M}:L^{\frac{p(\cdot )\tau (\cdot )}{rd}%
}\rightarrow L^{\frac{p(\cdot )\tau (\cdot )}{rd}}$ is bounded. The last
norm is bounded by $1$ if and only if%
\begin{equation*}
\Big\|\frac{\omega _{v}\ast f}{|P|^{\tau \left( \cdot \right) }}\chi _{3P}%
\Big\|_{p(\cdot )}\lesssim 1.
\end{equation*}%
{Notice that $3P=\cup _{h=1}^{3^{n}}P_{h}$, where $\{P_{h}\}_{h=1}^{3^{n}}$
are disjoint dyadic cubes with side length $l(P_{h})=l(P)$. Therefore $\chi
_{3P}=\sum_{h=1}^{3^{n}}\chi _{P_{h}}$ and 
\begin{eqnarray*}
\Big\|\frac{\omega _{v}\ast f}{|P|^{\tau \left( \cdot \right) }}\chi _{3P}%
\Big\|_{p(\cdot )} &\leq &c\sum_{h=1}^{3^{n}}\Big\|\frac{\omega _{v}\ast f}{%
|P_{h}|^{\tau \left( \cdot \right) }}\chi _{P_{h}}\Big\|_{p(\cdot )} \\
&\lesssim &\big\|\omega _{v}\ast f\big\|_{\widetilde{L_{\tau (\cdot
)}^{p(\cdot )}}} \\
&\lesssim &1.
\end{eqnarray*}%
Using a combination of the arguments used in the corresponding case of the
proof of }Lemma {\ref{Alm-Hastolemma1}} and those used in the estimate of $%
J_{v}^{1}$ above, we arrive at the desired estimate.\vskip5pt

{\emph{Proof of Lemma }\ref{lamda-equi}. Obviously, }%
\begin{equation*}
{\big\|\lambda \big\|_{\mathfrak{b}_{p\left( \cdot \right) ,q\left( \cdot
\right) }^{\alpha \left( \cdot \right) ,\tau (\cdot )}}\leq \big\|\lambda
_{r,d}^{\ast }\big\|_{\mathfrak{b}_{p\left( \cdot \right) ,q\left( \cdot
\right) }^{\alpha \left( \cdot \right) ,\tau (\cdot )}}.}
\end{equation*}%
{We will prove that }%
\begin{equation*}
{\big\|\lambda _{r,d}^{\ast }\big\|_{\mathfrak{b}_{p\left( \cdot \right)
,q\left( \cdot \right) }^{\alpha \left( \cdot \right) ,\tau (\cdot
)}}\lesssim \big\|\lambda \big\|_{\mathfrak{b}_{p\left( \cdot \right)
,q\left( \cdot \right) }^{\alpha \left( \cdot \right) ,\tau (\cdot )}}.}
\end{equation*}%
{\ For each $k\in \mathbb{N}_{0}$ define }%
\begin{equation*}
{\Omega _{k}:=\{h\in \mathbb{Z}^{n}:2^{k-1}<2^{v}\left\vert
2^{-v}h-2^{-v}m\right\vert \leq 2^{k}\}}
\end{equation*}%
{and }%
\begin{equation*}
{\Omega _{0}:=\{h\in \mathbb{Z}^{n}:2^{v}\left\vert
2^{-v}h-2^{-v}m\right\vert \leq 1\}.}
\end{equation*}%
{Then for any $x\in Q_{v,m}\cap P$, }%
\begin{equation}
{\sum_{h\in \mathbb{Z}^{n}}\frac{2^{vr\left( \alpha \left( x\right)
+n/2\right) }|\lambda _{v,h}|^{r}}{(1+2^{v}|2^{-v}h-2^{-v}m|)^{d}}}
\label{key-est1.2}
\end{equation}%
{can be rewritten as%
\begin{eqnarray}
&&\sum\limits_{k=0}^{\infty }\sum\limits_{h\in \Omega _{k}}\frac{2^{vr\left(
\alpha \left( x\right) +n/2\right) }\left\vert \lambda _{v,h}\right\vert ^{r}%
}{\left( 1+2^{v}\left\vert 2^{-v}h-2^{-v}m\right\vert \right) ^{d}}  \notag
\\
&\lesssim &\sum\limits_{k=0}^{\infty }2^{-dk}\sum\limits_{h\in \Omega
_{k}}2^{vr\left( \alpha \left( x\right) +n/2\right) }\left\vert \lambda
_{v,h}\right\vert ^{r}  \notag \\
&=&\sum\limits_{k=0}^{\infty }2^{(n-d)k+(v-k)n+vr\left( \alpha \left(
x\right) +n/2\right) }\int\limits_{\cup _{z\in \Omega
_{k}}Q_{v,z}}\sum\limits_{h\in \Omega _{k}}\left\vert \lambda
_{v,h}\right\vert ^{r}\chi _{v,h}(y)dy.  \label{key-est2}
\end{eqnarray}%
Let $x\in Q_{v,m}\cap P$ and $y\in \cup _{z\in \Omega _{k}}Q_{v,z}$. Then $%
y\in Q_{v,z}$ for some $z\in \Omega _{k}$ and $2^{k-1}<2^{v}\left\vert
2^{-v}z-2^{-v}m\right\vert \leq 2^{k}$. From this it follows that $y$ is
located in the cube $Q(x,2^{k-v+3})$. }Therefore, \eqref{key-est2} does not
exceed%
\begin{eqnarray*}
&&c\sum\limits_{k=0}^{\infty
}2^{(n-d+a)k+(v-k)n}\int\limits_{Q(x,2^{k-v+3})}2^{v(\alpha \left( y\right) +%
\frac{n}{2})r}\sum\limits_{h\in \Omega _{k}}\left\vert \lambda
_{v,h}\right\vert ^{r}\chi _{v,h}(y)dy \\
&=&c\sum\limits_{k=0}^{\infty }2^{(n-d+a)k}M_{Q(x,2^{k-v+3})}\left(
g_{v}\right)
\end{eqnarray*}%
for some positive constant $c$ independent of $v$ and $k$ and 
\begin{equation*}
g_{v}=2^{v(\alpha \left( \cdot \right) +\frac{n}{2})r}\sum\limits_{h\in 
\mathbb{Z}^{n}}\left\vert \lambda _{v,h}\right\vert ^{r}\chi _{v,h},\quad
v\geq v_{P}^{+}.
\end{equation*}%
Observe that 
\begin{equation*}
M_{Q(x,2^{k-v+3})}\left( g_{v}\right) \lesssim 2^{kL}\eta _{v,L}\ast g_{v}(x)
\end{equation*}%
for any $x\in Q_{v,m}\cap P${\ and any }$L>n$ large enought, where the
implicit constant is indepenendt of $x,k$ and $v${. Therefore }%
\eqref{key-est1.2} is bounded by 
\begin{equation*}
c\text{ }\eta _{v,L}\ast g_{v}(x),\quad x\in {Q_{v,m}\cap P.}
\end{equation*}%
Thanks to Lemma {\ref{Alm-Hastolemma1}, we have}%
\begin{eqnarray*}
{\big\|\lambda _{r,d}^{\ast }\big\|_{\mathfrak{b}_{p\left( \cdot \right)
,q\left( \cdot \right) }^{\alpha \left( \cdot \right) ,\tau (\cdot )}}} &{%
\lesssim }&\big\|(\eta _{v,L}\ast g_{v})_{v}\big\|_{\ell ^{\frac{q(\cdot )}{r%
}}(L_{\frac{p(\cdot )}{r}}^{r\tau (\cdot )})}^{\frac{1}{r}} \\
&{\lesssim }&\big\|(g_{v})_{v}\big\|_{\ell ^{\frac{q(\cdot )}{r}}(L_{\frac{%
p(\cdot )}{r}}^{r\tau (\cdot )})}^{\frac{1}{r}} \\
&{\lesssim }&{\big\|(}\lambda _{v})_{v}{\big\|_{\mathfrak{b}_{p\left( \cdot
\right) ,q\left( \cdot \right) }^{\alpha \left( \cdot \right) ,\tau (\cdot
)}},}
\end{eqnarray*}%
provided that $d$ is sufficiently large such that $d>n+a+L$. The proof of
the lemma is thus complete.

\bigskip

Douadi \ Drihem and Zeghad Zouheyr

M'sila University

Department of Mathematics

Laboratory of Functional Analysis and Geometry of Spaces

P.O. Box 166, M'sila 28000, Algeria

E-mail:\texttt{douadidr@yahoo.fr, douadi.drihem@univ-msila.dz, (Douadi
Drihem)}

\texttt{zgzoheyr@gmail.com, zouheyr.zeghad@univ-msila.dz (Zeghad Zouheyr)}

\end{document}